\documentclass[11pt,eqno]{article}
\usepackage{bbm}
\usepackage[reqno]{amsmath}
\usepackage{amsfonts}
\usepackage{graphicx}
\usepackage{amsmath}
\usepackage[pagewise]{lineno}
 \usepackage{paralist}
\usepackage{amssymb}
\usepackage{latexsym}
\usepackage{amsmath, amsfonts,amssymb, amsthm, euscript,makeidx,color,mathrsfs}
\usepackage{enumerate}
\usepackage{subfigure}
\usepackage{ulem}

\usepackage{autobreak}
\allowdisplaybreaks[4]

\usepackage{xcolor}
\usepackage[colorlinks,linkcolor=blue,anchorcolor=green,citecolor=red]{hyperref}

\DeclareMathOperator*{\supp}{\mathrm{supp}}
\DeclareMathOperator*{\dist}{\mathrm{dist}}

\def\mb{\mathbb}

\def\Om{\Omega}
\def\ra{\rightarrow}

\def\pt{\partial}
\def\mcH{\mathcal{H}}

\newtheorem{lemma}{Lemma}[section]
\newtheorem{theorem}[lemma]{Theorem}
\newtheorem{remark}[lemma]{Remark}

\theoremstyle{definition}
\newtheorem{definition}[lemma]{Definition}

\makeatletter

\@addtoreset{equation}{section}
\makeatother

\begin{document}
\title{Extremal properties of the first eigenvalue and the fundamental gap of a sub-elliptic operator}

\date{}

\author{\sffamily Hongli Sun$^{1}$, Weijia Wu$^{1*}$, Donghui Yang$^1$ \\
	{\sffamily\small $^1$ School of Mathematics and Statistics, Central South University, Changsha 410083, China. }\\
}

\footnotetext[1]{Corresponding author: weijiawu@yeah.net}

\maketitle
\begin{center}
\begin{abstract}
 We consider the problems of extreming the first eigenvalue and the fundamental gap of a sub-elliptic operator with Dirichlet boundary condition, when the potential $V$ is subjected to a $p$-norm constraint.  The existence results for weak solutions, compact embedding theorem and spectral theory for sub-elliptic equation are given.  Moreover, we provide the specific characteristics of the corresponding optimal potential function.

 \vspace{0.2cm}

  \vspace{0.2cm}  


\end{abstract}

\end{center}

\section{Introduction}
For the extremal properties of the first eigenvalue and the fundamental gap, it seems to be a broader focus on  Schr{\"o}dinger equation. In general, consider the operator $-\Delta+V(x)$ with the Dirichlet boundary condition, where $V$ is a potential, a multiplication operator, $x \in \Omega \subset \mathbb{R}^n$, $\Omega$ is a fixed bounded smooth domain. For reasonable potentials $V$, the spectrum is discrete and consists of non-negative eigenvalues which can be numbered in an increasing order as follows
\begin{equation}\label{202212181021}
	0<\lambda_1<\lambda_{2}\le \lambda_3\cdots\le \lambda_{m}\le \cdots.
\end{equation}
In quantum mechanics, these eigenvalues correspond to the energy levels, in atomic units, of a quantum particle in the potential energy $V$ imagined as $+\infty$ outside $\Omega$.  The first eigenvalue $\lambda_{1}(V)$ is generally referred as the ground state, the second $\lambda_{2}(V)$ is the first excited state and the difference $\Gamma(V)=\lambda_{2}(V)-\lambda_{1}(V)$ is the fundamental gap.

The fundamental gap typically has profound physical implications and mathematical senses. Simultaneously, the gap is also a core topic of interest in statistical mechanics and quantum field theory, especially,  if it is small enough of the size for the fundamental gap, it can produce the well-known tunnelling effect, hence, the problem of minimizing the fundamental gap seems very crucial. Moreover, the gap plays an important role in both numerical calculation and analysis, such as in the improvement of Poincar\'{e} inequality, a priori estimates. For more information, we can refer to \cite{ashbaugh2006fundamental,Henrot2006Extremum}.

The problem of maximizing the first eigenvalue of a Schr{\"o}dinger operator among potentials $V$ of given $L^p$ norm was initiated at least in the early 1960s \cite{keller1961lower}.  Subsequently, in one-dimensional case, detailed descriptions were given on \cite{essen1987estimating, talenti1984estimates,veling1982optimal,bennewitz1992optimal,karaa1998sharp}, in particular, the eigenvalue problem of unbounded interval was considered in \cite{ veling1982optimal,bennewitz1992optimal}. In a situation of multidimension, \cite{harrell1984hamiltonian,egnell1987extremal,ashbaugh1987maximal} provided the effective solution to ideas and methods for maximizing the first eigenvalue. The problem of extreming eigenvalues of
Schr{\"o}dinger operators on manifolds was considered, for example in \cite{Soufi2006critical,seo2014lp,freitas2001minimal}. For other types of operators, the interesting problem about the eigenvalue was also discussed, for instance, the Dirichlet eigenvalue problem for degenerate elliptic operator $-\Delta_X$ on $\Omega$ was considered on \cite{chenhualuo2015lower}, where $X=(X_1,X_2,\cdots, X_m)$  be a system of real smooth vector fields defined on $\Omega$ with the boundary $\partial{\Omega}$ which is non-characteristic for $X$. We also refer to \cite{meng2016optimal,brandolini2015optimal}  for similar topics.

Van den Berg put forward in 1983 \cite{van1983condensation} that the fundamental gap $\Gamma(V)$ is bounded below $3\pi^2/d^2$, where $d$ is the diameter of the convex domain.  One can find more results in \cite{ashbaugh2006fundamental}.   Beyond the work \cite{van1983condensation}, the recent history of the fundamental gap problem
mainly comes from \cite{singer1985estimate}, which obtained that $\Gamma(V)\ge \pi^2/4d^2$ of the general Schr{\"o}dinger problem for all dimensions $n$. This result has been improved in many subsequent works \cite{yu1986lower,wang2000estimation,ling2008estimates},  it was not until
2011 that Andrews and Clutterbuck \cite{Ben2011Proof} completely solved this conjecture. For one-dimensional case, there are also many excellent works, we can refer to \cite{ashbaugh1989optimal,horvath2003first,huang2009eigenvalue,el2022optimal,kerner2022lower,ashbaugh2020spectral}. Notably, the problem of extreming the fundamental gap of a Schr{\"o}dinger operator among potentials $V$ of given $L^p$ norm was studied on \cite{Ashbaugh1991ON,karaa1998extremal}. For more general operators, the gap of a second order self-adjoint operator was considered in a domain \cite{borisov2012spectral}, the boundary is partitioned into two parts with Dirichlet boundary condition on one of them, and Neumann condition on the other one. Wolfson \cite{Wolfson2015eigenvalue} gived a estimate the eigenvalue gap for a class of nonsymmetric second-order linear elliptic operators. The eigenvalue gap of the $p$-Laplacian eigenvalue problem was introduced in \cite{cheng2014dual}. 

Inspired by these works, in this paper we consider the extremal properties of the eigenvalue problem for the following degenerate sub-elliptic equation:
\begin{equation}\label{202301172048}
	\begin{cases}
		-\left(y^{2\alpha_1}\partial_{x}^2u+x^{2\alpha_2}\partial_{y}^2u\right)+Vu= \lambda u, & x\in \Omega, \\
		u=0, & x\in\partial \Omega, 
	\end{cases} 
\end{equation}
where $\alpha_1,\alpha_2\in (0,\frac{1}{2})$,
\begin{equation*}
	\Omega=\left\{(x,y)\in \mathbb{R}^2\mid 0<x<1, 0<y<1\right\},
\end{equation*}
and 
\begin{equation*}
V\in S_p=\left\{V\in L^{p}(\Omega)\left| \ \! V\ge 0 \ a.e., \|V\|_{L^p(\Omega)} \le M\right.\right\},
\end{equation*} 
$M$ is a positive constant, $ 1< p< \infty$. We shall discuss the solvability of the differential equation, and prove the corresponding compact embedding theorem. Based on these, the characteristics of spectrum are described as \eqref{202212181021}, which shows that it is meaningful to study the fundamental gap. The aim of the paper is to find optimal potentials associated to $\sup\limits_{V\in S_p}\lambda_{1}(V)$ and $\inf\limits_{V\in S_p}\Gamma(V)$ for some $p$. 

In this paper, we proceed as follows. In Section 2, we introduce weak solution space and provide the 
corresponding compact embedding theorem.  The local boundedness and regularity of the weak solution are established in Section 3.  In Section 4, we consider the maximum of the first eigenvalue problem \eqref{202301172048} when $V \in S_p$, $1<p<\infty$. In Section 5, the minimum problem of fundamental gap is characterized when $V \in S_p$, $2<p<\infty$.

\section{Solution space}
Our overall plan is first to define and then construct proper weak solution $u$  with respect to \eqref{202301231543} and only later to explore other properties of $u$.  The space we mentioned below may be involved in many works such as \cite{2009Maximum,sawyer2010degenerate,cavalheiro2008weighted,auscher2007maximal, kufner1984define}, but we will introduce more detailed results. Consider
\begin{equation}\label{202301231543}
	\begin{cases}
		-\left(y^{2\alpha_1}\partial_{x}^2u+x^{2\alpha_2}\partial_{y}^2u\right)+Vu= f, & x\in \Omega, \\
		u=0, & x\in\partial \Omega,
	\end{cases} 
\end{equation}
where $f \in L^2(\Omega)$, $\alpha_1,\alpha_2 \in (0,\frac{1}{2})$. 

Define
\begin{equation*}
	\mcH^1(\Omega)=\left\{u\in L^2(\Omega)\Bigm| y^{\alpha_1}\partial_{x}u\in L^2(\Omega), x^{\alpha_2}\partial_{y}u\in L^2(\Omega)\right\}, 
\end{equation*}
with scalar product
\begin{equation*}
	(u,v)_{\mcH^1(\Omega)}=\iint_\Omega uv +(y^{\alpha_1}\pt_{x}u)(y^{\alpha_1}\pt_{x}v)+\left(x^{\alpha_2}\pt_{y}u\right)\left(x^{\alpha_2}\pt_{y}v\right)dxdy,
\end{equation*}
endowed with the norm
\begin{equation*}
	\|u\|_{\mcH^1(\Omega)}=\left(\|u\|_{L^2(\Omega)}^2+\|y^{\alpha_1}\pt_{x}u\|_{L^2(\Omega)}^2+\|x^{\alpha_2}\pt_{y}u\|_{L^2(\Omega)}^2\right)^\frac{1}{2}. 
\end{equation*}
It is not hard to find that $\|u\|_{\mcH^1(\Omega)}=\sqrt{(\cdot,\cdot)_{\mcH^1(\Omega)}}$.

Throughout this paper,  let  $\|\cdot\|$ and  $(\cdot,\cdot)$ denote the norm and the inner product of $L^2(\Omega)$,  $H^1(\Omega)$ is the classical Sobolev space.

Let $\mcH_0^1(\Omega)$ denotes the closure of $C_0^\infty(\Omega)$ in the space $\mcH^1(\Omega)$, that is, 
\begin{equation*}
	\mcH_0^1(\Omega)=\overline{C_0^\infty(\Omega)}^{\mcH^1(\Omega)}. 
\end{equation*}

\begin{lemma}\label{Lem.2.1}\label{202210162153}	
	The space $(\mcH^1(\Omega),(\cdot,\cdot)_{\mcH^1(\Omega)})$ is a Hilbert space.   
\end{lemma}

\begin{proof}	
	Firstly, we easily verify that $(\mcH^1(\Omega),(\cdot,\cdot)_{\mcH^1(\Omega)})$ is an inner space. Let $\{u_n\}_{n\in\mb{N}}\subset\mcH^1(\Omega)$ be a Cauchy sequence, so that
	$\{u_n\}_{n\in\mb{N}},\{y^{\alpha_1}\pt_{x}u_n\}_{n\in\mb{N}},\{x^{\alpha_2}\pt_{y}u_n\}_{n\in\mb{N}}$  are Cauchy sequences in $L^2(\Omega)$.  Then there exist $u,v,w\in L^2(\Omega)$ such that 
	\begin{equation*}
		u_n\ra u, y^{\alpha_1}\pt_{x}u_n\ra v, x^{\alpha_2}\pt_{y}u_n\ra w \mbox{ strongly in } L^2(\Omega). 
	\end{equation*}
	For each $\varphi\in C_0^\infty(\Omega)$, we have 
	\begin{equation*}
		\begin{split}
			\iint_\Omega v\varphi dxdy \leftarrow \iint_\Omega \left(y^{\alpha_1}\pt_{x}u_n\right)\varphi dxdy =-\iint_\Omega u_n\left(y^{\alpha_1}\pt_{x}\varphi\right) dxdy \ra -\iint_\Omega  u\left({y^{\alpha_1}}\pt_{x}\varphi\right)dxdy, \ n\to \infty
		\end{split}
	\end{equation*}
	in the sense of distribution, which implies that 
	\begin{equation*}
	v=y^{\alpha_1}\pt_{x}u
	\end{equation*}
	in the sense of distribution. Naturally, $v=y^{\alpha_1}\pt_{x}u$ in $L^2(\Omega)$ since $w\in L^2(\Omega)$. Implement the same method again then $w=x^{\alpha_2}\pt_{y}u$ in $L^2(\Omega)$ is attained. All these imply that 
	\begin{equation*}
		u_n\ra u \mbox{  in }\mcH^1(\Omega).\qedhere
	\end{equation*}
\end{proof}

\begin{lemma}\label{Lem.20221006}
	The space $\mcH^1(\Omega)$ is separable and reflexive.  
\end{lemma}

\begin{proof}
	Define $L_3^2(\Omega)=L^2(\Omega) \times L^2(\Omega)\times L^2(\Omega)$, where
	\begin{equation*}
		\|u\|_{L_3^2(\Omega)}=\left(\iint_\Omega |u_1|^2+ |u_2|^2+  |u_3|^2dxdy \right)^{\frac{1}{2}}
	\end{equation*}
	for $u=(u_1,u_2,u_3) \in L_3^2(\Om)$ as the norm of space 	$L_3^2(\Omega)$. It is evident that $L_3^2(\Omega)$ is a separable space.  Set
	\begin{equation*}
		Pu=(u, x^{\alpha_2}\pt_{y}u,y^{\alpha_1}\pt_{x}u),u\in \mcH^1(\Omega),
	\end{equation*}	
	clearly, $W=\{Pu\mid u\in \mcH^1(\Omega)\}$ is a subspace of $L_3^2(\Omega)$.  It is found that  $P$ is an isometric isomorphism of mapping $\mcH^1(\Omega)$ to $W$ in view of $\|Pu\|_{L_3^2(\Omega)}=\|u\|_{\mcH^1(\Omega)}$. Given that $\mcH^1(\Omega)$ is complete,  $W$ is a closed subspace of $L_3^2(\Omega)$, it implies that $W$ is separable. Note that $P$ is an isometric linear isomorphism, then $\mcH^1(\Omega)$ has the same properties. The reflexivity can be obtained in the same way.
\end{proof}

\begin{definition}
We define the space
	\begin{equation*}
	H(\Omega)=\left\{u\in \mcH^1(\Omega) \left|\ \! \iint_\Omega Vu^2dxdy < \infty \right.\right\},
\end{equation*}	
with scalar product
	\begin{equation*}
	(u,v)_{H(\Omega)}=\iint_\Omega uv +(y^{\alpha_1}\pt_{x}u)(y^{\alpha_1}\pt_{x}v)+\left(x^{\alpha_2}\pt_{y}u\right)\left(x^{\alpha_2}\pt_{y}v\right)+(V^\frac{1}{2}u)(V^\frac{1}{2}v)dxdy,
\end{equation*}
and the associated norm
\begin{equation*}
	\|u\|_{H(\Omega)}=\left(\|u\|_{L^2(\Omega)}^2+\|y^{\alpha_1}\pt_{x}u\|_{L^2(\Omega)}^2+\|x^{\alpha_2}\pt_{y}u\|_{L^2(\Omega)}^2+ \|V^\frac{1}{2}u\|_{L^2(\Omega)}^2\right)^\frac{1}{2}.
\end{equation*}
We also define the space $H_0(\Omega)=\left\{u\in \mcH_0^1(\Omega) \left|\ \! \iint_\Omega Vu^2dxdy < \infty \right.\right\}$. 
\end{definition}

\begin{lemma}\label{Lem.202210162141}
	The space$(H(\Omega),(\cdot,\cdot)_{H(\Omega)})$ is a  Hilbert space.  
\end{lemma}

\begin{proof}
Let $\{u_n\}_{n\in\mb{N}}\subset H(\Omega)$ be a Cauchy sequence, thus $\{V^\frac{1}{2}u_n\}_{n\in\mb{N}}$  is a Cauchy sequence in $L^2(\Omega)$, i.e., there exists $g \in L^2(\Omega)$ such that 
\begin{equation*}
	V^\frac{1}{2}u_n \to g \ \mbox{ in } \ L^2(\Omega).	
\end{equation*}
From Lemma \ref{202210162153} it follows that $\|u_n-u\|_{\mcH^1(\Omega)} \to 0$, as $n \to \infty$, it is sufficient to prove that
\begin{equation}\label{202210170933}
g= V^\frac{1}{2}u \ \mbox{  in } \ L^2(\Omega).	
\end{equation}
Since $\{V^\frac{1}{2}u_n\}_{n\in\mb{N}}$ is a Cauchy sequence, one has for $\forall \epsilon>0$, there exists $N \in \mathbb{N}$, such that $\|V^\frac{1}{2}u_m-V^\frac{1}{2}u_n\|_{L^2(\Omega)}< \epsilon$ for any $m, n \ge N$. Moreover, there is a subsequence of $\{u_n\}_{n\in\mb{N}}$,  denoted by $\{u_{n_j}\}_{j\in\mb{N}}$, such that 
   \begin{equation*}
   	u_{n_j} \rightarrow u	\ \mbox{ a.e. as} \ j \rightarrow \infty,
   \end{equation*}
given that $u_n \to u$ in $L^2(\Omega)$. This implies that $V^\frac{1}{2}u_{n_j} \to V^\frac{1}{2}u$ a.e.\! as  $j \rightarrow \infty$. Fixed $n\ge N$,  applying Fatou Theorem
 \begin{equation*}
 	\begin{split}
 	\iint_{\Omega} \left|V^\frac{1}{2}u-V^\frac{1}{2}u_n\right|^2dxdy&=\iint_{\Omega} \lim\limits_{j \to \infty} \left|V^\frac{1}{2}u_{n_j}(x)-V^\frac{1}{2}u_n(x)\right|^2dxdy\\	
 	&\le \varliminf\limits_{j \to \infty} \iint_{\Omega} \left|V^\frac{1}{2}u_{n_j}-V^\frac{1}{2}u_n\right|^2dxdy\le \epsilon^2.
 	\end{split}
\end{equation*}
So far, \eqref{202210170933} is verified.	
\end{proof}

\begin{lemma}\label{202301231730}
	$\mcH^1(\Omega)$ is continuously embedded into $W^{1,1}(\Omega)$.
\end{lemma}

\begin{proof}
For any $u \in \mcH^1(\Omega)$,	using H{\"o}lder inequality,
	\begin{equation*}
		\begin{split}		
		\iint_{\Omega} |\partial_x u|dxdy& \le \left(\iint_{\Omega} y^{-2\alpha_1} dxdy\right)^{\frac{1}{2}} \left(\iint_{\Omega} |y^{\alpha_1}\partial_x u|^2 dxdy\right)^{\frac{1}{2}} \\
		&\le C \|y^{\alpha_1}\partial_x u\|_{L^2(\Omega)}.
	    \end{split}
	\end{equation*}
The same procedure is implemented again, we have $\|\partial_y u\|_{L^1(\Omega)} \le C \|x^{\alpha_2}\partial_y u\|_{L^2(\Omega)}$.		 
\end{proof}

\begin{lemma}\label{202301232024}
$H_0^1(\Omega) \subset \mcH_0^1(\Omega) \subset \mcH^1(\Omega) \cap W_0^{1,1}(\Omega)$.
\end{lemma}

\begin{proof}
	For $u \in H_0^1(\Omega)$, there is a sequence $\{u_n\}_{n\in\mb{N}} \in C_0^\infty(\Omega)$ such that $\|u_n-u\|_{H^1(\Omega)} \to 0$, $n \to \infty$, and we observe that 
	\begin{equation*}
		\|u_n-u\|_{\mcH^1(\Omega)} \le C\|u_n-u\|_{H^1(\Omega)} \to 0, \ n \to \infty,
	\end{equation*}	
	therefore $u \in \mcH_0^1(\Omega)$.

	For $u \in \mcH_0^1(\Omega)$, there is a sequence $\{u_n\}_{n\in\mb{N}} \in C_0^\infty(\Omega)$ such that $\|u_n-u\|_{\mcH^1(\Omega)} \to 0$, $n \to \infty$. And from Lemma \ref{202301231730}
	\begin{equation*}
		\|u_n-u\|_{W^{1,1}(\Omega)} \le C\|u_n-u\|_{\mcH^1(\Omega)} \to 0, \ n \to \infty,
    \end{equation*}	
so that we have $ \mcH_0^1(\Omega) \subset \mcH^1(\Omega) \cap W_0^{1,1}(\Omega)$.		
\end{proof}

From Lemma \ref{202301232024}, we know that it makes sense to consider Dirichlet condition problem \eqref{202301231543}.

\begin{lemma}\label{202301160920}
For any $u\in C_0^\infty(\Omega)$, we have 
\begin{equation}\label{202302192045}
\|u\|_{L^\frac{2\delta}{1-\delta}(\Omega)} \le\left( \iint_\Omega |y^{\alpha_1}\partial_xu|^2   +  |x^{\alpha_2}\partial_yu|^2dxdy  \right)^\frac{1}{2}	
\end{equation}
for $\frac{1}{2}\le\delta<\min\left\{\frac{1}{\alpha_1+1},\frac{1}{\alpha_2+1}\right\}$.	
\end{lemma}

\begin{proof}
For any $u\in C_0^\infty(\Omega)$, we observe that
\begin{equation*}
	u(x,y)=\int_0^{x}\partial_x(s,y)ds,
\end{equation*}	
then
\begin{equation}\label{20220925-1}
	|u(x,y)|\le \int_0^1 \left|\partial_xu(s,y)\right| ds.
\end{equation}	

Similarly, we have $|u(x,y)|\le \int_0^1 |\partial_yu(x,t)| dt$, furthermore, for $0<\delta<1$,

\begin{equation}\label{20220925-2}
|u(x,y)|^\delta   \le   \left(\int_0^1 \big|\partial_x(s,y)\big| ds \right)^\delta ,\quad |u(x,y)|^\delta   \le   \left(\int_0^1 \big|\partial_y(x,t)\big| dt \right)^\delta .
\end{equation}	
Therefore
\begin{align*}
&\int_0^1 \int_0^1	|u(x,y)|^{2\delta}dx dy \\
&  \le   \int_0^1 \int_0^1  \left\{ \left(\int_0^1 \big|\partial_x(s,y)\big| ds \right)^\delta   \left(\int_0^1 \big|\partial_y(x,t)\big| dt \right)^\delta                 \right\}dxdy\\
&=\int_0^1 \left(\int_0^1 \big|\partial_y(x,t)\big| dt \right)^\delta \left\{\int_0^1 \left(\int_0^1 \big|\partial_x(s,y)\big| ds \right)^\delta dy                \right\} dx\\
&=\int_0^1 \left(\int_0^1 \big|\partial_x(s,y)\big| ds \right)^\delta dy  \cdot       \int_0^1 \left(\int_0^1 \big|\partial_y(x,t)\big| dt \right)^\delta dx\\
&=\int_0^1 \left(\int_0^1 \big|\partial_xu\big|dx\right)^\delta  y^{-\alpha_1\delta} y^{\alpha_1\delta}  dy  \cdot\int_0^1 \left(\int_0^1 \big|\partial_yu\big| dy \right)^\delta x^{-\alpha_2\delta}  x^{\alpha_2 \delta}   dx\\
&\le \left\{\int_0^1 y^{-\frac{\alpha_1\delta}{1-\delta}}dy  \right\}^{1-\delta} 
\left\{\int_0^1 y^{\alpha_1}\int_0^1 \big|\partial_xu\big| dy dx    \right\}^\delta\cdot  \left\{\int_0^1 x^{-\frac{\alpha_2\delta}{1-\delta}}dx  \right\}^{1-\delta} 
 \left\{\int_0^1 x^{\alpha_2}\int_0^1 \big|\partial_yu\big| dy dx    \right\}^\delta.
\end{align*}

Let $\frac{1}{2}\le\delta<\min\left\{\frac{1}{\alpha_1+1},\frac{1}{\alpha_2+1}\right\}$, we have
\begin{equation*}
	\int_0^1 \int_0^1 |u(x,y)|^{2\delta}dxdy \le C \|y^{\alpha_1} \partial_xu\|_{L^1(\Omega)}^\delta \|x^{\alpha_2} \partial_yu\|_{L^1(\Omega)}^\delta,
\end{equation*}
that is,
\begin{equation}\label{20220925a1}
	\begin{split}
	\|u\|_{L^{2\delta}(\Omega)}
	&\le C \|y^{\alpha_1}\partial_x u\|_{L^1(\Omega)}^{\frac{1}{2}} \|x^{\alpha_2} \partial_yu\|_{L^1(\Omega)}^{\frac{1}{2}}\le C\left( \|y^{\alpha_1}\partial_x u\|_{L^1(\Omega)}+  \|x^{\alpha_2} \partial_yu \|_{L^1(\Omega)} \right).
\end{split} 
\end{equation}

We put  $|u|^\gamma$ $(\gamma>1)$ into \eqref{20220925a1} and obtain
\begin{equation}\label{20220925a2}
	\begin{split}
	\left\| |u|^\gamma\right\|_{L^{2\delta}(\Omega)}& \le C \left(\left \||u|^{\gamma-1} y^{\alpha_1}\partial_xu\right\|_{L^1(\Omega)}+\left\| |u|^{\gamma-1} x^{\alpha_2} \partial_yu\right\|_{L^1(\Omega)}\right)\\
	& \le C  \left\||u|^{\gamma-1}\right\|_{L^2(\Omega)}  \left( \left\|y^{\alpha_1}\partial_xu\right  \|_{L^2(\Omega)} +  \left\|x^{\alpha_2}\partial_yu\right\|_{L^2(\Omega)}                 \right).
\end{split} 
\end{equation}

Taking $\gamma=\frac{1}{1-\delta}$, 
\begin{equation*}
\left\||u|^\gamma\right\|_{L^{2\delta}(\Omega)}=\left(\|u\|_{L^\frac{2\delta}{1-\delta}(\Omega)}\right)^{\frac{1}{1-\delta}}, \quad 	\left\||u|^{\gamma-1}\right\|_{L^2(\Omega)}=\left(\|u\|_{\frac{2\delta}{1-\delta}(\Omega)}\right)^\frac{\delta}  {1-\delta},			
\end{equation*}
 the inequality \eqref{20220925a2}  becomes: 
\begin{equation*}
\left(||u||_{L^\frac{2\delta}{1-\delta}(\Omega)}\right)^{\frac{1}{1-\delta}}    \le C	\left(\|u\|_{L^\frac{2\delta}{1-\delta}(\Omega)}\right)^\frac{\delta}  {1-\delta}	\left( \|y^{\alpha_1}\partial_xu \|_{L^2(\Omega)} +  \|x^{\alpha_2}\partial_yu\|_{L^2(\Omega)}                 \right),		
\end{equation*}
therefore 
\begin{equation*}
	\begin{split}
		\|u\|_{L^\frac{2\delta}{1-\delta}(\Omega)}	&\le C \left( \|y^{\alpha_1}\partial_xu  \|_{L^2(\Omega)} +  \|x^{\alpha_2}\partial_yu\|_{L^2(\Omega)}  \right)\\
		&\le C	\left( \iint_\Omega |y^{\alpha_1}\partial_xu|^2   +  |x^{\alpha_2}\partial_yu|^2dxdy  \right)^\frac{1}{2}.		
	\end{split}	
\end{equation*}	
\end{proof}

\begin{remark}\label{202209261124}
 We easily obtain that the inequality \eqref{202301160920} is established for all $u\in \mcH_0^1(\Omega)$ when $\frac{1}{2}\le\delta<\min\left\{\frac{1}{\alpha_1+1},\frac{1}{\alpha_2+1}\right\}$.  
 Moreover, the Poincar\'e inequality is received by taking $\delta=\frac{1}{2}$. Hence, the equivalent norm of $\mcH_0^1(\Omega)$ and $H_0(\Omega)$ are
\begin{equation*}
	\begin{split}
		\|u\|_{\mcH_0^1(\Omega)}&=\left(\|y^{\alpha_1}\pt_{x}u\|_{L^2(\Omega)}^2+\|x^{\alpha_2}\pt_{y}u\|_{L^2(\Omega)}^2\right)^\frac{1}{2},\\
		\|u\|_{H_0(\Omega)}&=\left(\|y^{\alpha_1}\pt_{x}u\|_{L^2(\Omega)}^2+\|x^{\alpha_2}\pt_{y}u\|_{L^2(\Omega)}^2+ \|V^\frac{1}{2}u\|_{L^2(\Omega)}^2\right)^\frac{1}{2}.	
	\end{split}
\end{equation*}
\end{remark}

\begin{theorem}\label{202210121135}
$\mcH_0^1(\Omega)$ is compactly embeded in $L^m(\Omega)$, where $m\in [1,4)$.
\end{theorem}

\begin{proof}
	Since $W^{1,1}(\Omega)$ is compactly embeded into $L^1(\Omega)$, we find that the embedding $\mcH_0^1(\Omega)\hookrightarrow L^1(\Omega)$ is compact by Lemma \ref{202301231730}. In other words,
	let $\{u_n\}_{n\in\mb{N}}\subset\mcH_0^1(\Omega)$ be a bounded sequence with a upper bound $M_1$,   there exists a subsequence of $\{u_n\}_{n\in\mb{N}}$, still denoted by itself,  converge  in $L^1(\Omega)$.

By Lemma \ref{202301160920}, we obtain that $\{u_n\}_{n\in\mb{N}}$ is bounded in $L^{4}(\Omega)$ since $\min\left\{\frac{1}{\alpha_1+1},\frac{1}{\alpha_2+1}\right\} > \frac{2}{3}$. Through the interpolation inequality for $m \in (1,4)$, we have  
\begin{equation*}
\begin{split}
		\|u_n-u_m\|_{L^m(\Omega)}&\le \|u_n-u_m\|_{L^1(\Omega)}^\alpha \|u_n-u_m\|_{L^{4}(\Omega)}^{1-\alpha} \le (2M_1)^{1-\alpha} \|u_n-u_m\|_{L^1(\Omega)}^\alpha
\end{split}	
	\end{equation*}
for some $\alpha \in (0,1)$, given that $\{u_n\}_{n\in\mb{N}}$ converge in $L^1(\Omega)$, choosing $n,m$ sufficiently large shows that $\{u_n\}_{n\in\mb{N}}$ converge in $L^m(\Omega)$.
\end{proof}

\begin{remark}\label{202211051625}
	Since the  embedding 
	\begin{equation*}
		H_0(\Omega) \hookrightarrow \mcH_0^1(\Omega)
	\end{equation*}
is continuous, we have that the embedding
	\begin{equation*}
	H_0(\Omega) \hookrightarrow L^2(\Omega)
\end{equation*}
is compact by invoking the Theorem \ref{202210121135}. 
\end{remark}

\begin{definition}\label{202301241118}
	The {\it bilinear form} $a(\cdot,\cdot)$ associated with the sub-elliptic operator $L$ is 
	\begin{equation*}
		a(u,v)=\iint_\Omega(y^{\alpha_1}\pt_{x}u)(y^{\alpha_1}\pt_{x}v)+\left(x^{\alpha_2}\pt_{y}u\right)\left(x^{\alpha_2}\pt_{y}v\right)+ Vuv dxdy
	\end{equation*}
	for $u,v\in H_0(\Omega)$. We call $u \in H_0(\Omega)$ is a {\it weak solution} for the problem \eqref{202301231543} if
	\begin{equation}\label{202210282232}
		a(u,v)=(f, v)_{L^2(\Omega)}
	\end{equation}
	for all $v\in H_0(\Omega)$.
\end{definition}

\begin{theorem}\label{202301161746}
The boundary value problem \eqref{202301231543} has a unique weak solution $u\in H_0(\Omega)$ for any given $f\in L^2(\Omega)$.
\end{theorem}
\begin{proof}
For any $v \in H_0(\Omega)$, we deduce that $f: H_0(\Omega)\ra \mb{R}$ is a linear continuous functional on $H_0(\Omega)$ based on the fact that
  \begin{equation*}
 (f,v)_{L^2(\Omega)}\le \|f\|_{L^2(\Omega)}\|v\|_{L^2(\Omega)} \le  \|f\|_{L^2(\Omega)}\|v\|_{H_0(\Omega)}.
 \end{equation*}	
Thanks to the H\"older inequality, we observe that
\begin{align*}
		a(u,v)&=\iint_\Omega(y^{\alpha_1}\pt_{x}u)(y^{\alpha_1}\pt_{x}v)+\left(x^{\alpha_2}\pt_{y}u\right)\left(x^{\alpha_2}\pt_{y}v\right)+ Vuv dxdy,\\
		&\le\|y^{\alpha_1}\pt_{x}u\|_{L^2(\Omega)}\|y^{\alpha_1}\pt_{x}v\|_{L^2(\Omega)}+\|x^{\alpha_2}\pt_{y}u\|_{L^2(\Omega)}\|x^{\alpha_2}\pt_{y}v\|_{L^2(\Omega)}+ \|V^\frac{1}{2}u\|_{L^2(\Omega)}\|V^\frac{1}{2}v\|_{L^2(\Omega)}\\
		&\le \left(\|y^{\alpha_1}\pt_{x}u\|_{L^2(\Omega)}^2+\|x^{\alpha_2}\pt_{y}u\|_{L^2(\Omega)}^2+ \|V^\frac{1}{2}u\|_{L^2(\Omega)}^2 \right)^\frac{1}{2}\\
		&\hspace{6mm}\cdot \left(\|y^{\alpha_1}\pt_{x}v\|_{L^2(\Omega)}^2+\|x^{\alpha_2}\pt_{y}v\|_{L^2(\Omega)}^2 +\|V^\frac{1}{2}v\|_{L^2(\Omega)}^2 \right)^\frac{1}{2}\\
		&\le C \|u\|_{H_0(\Omega)}\|v\|_{H_0(\Omega)}	
\end{align*}	
for any $u, v \in H_0(\Omega)$. Additionally, for any $u \in H_0(\Omega)$
\begin{equation}
a(u,u)=\iint_\Omega |y^{\alpha_1}\pt_{x}u|^2+|x^{\alpha_2}\pt_{y}u|^2+ |V^{\frac{1}{2}}u|^2dxdy \ge \frac{1}{2}\|u\|_{H_0(\Omega)}^2.
\end{equation}
The desired result is proved by employing the Lax Milgram theorem.
\end{proof}

\section{ Local properties of weak solutions}

We have established the existence and uniqueness of the solution of problem \eqref{202301231543} and explored the compact embedding theorem. The purpose of this section is to reveal that weak solutions of equation \eqref{202301231543} are locally bounded and locally H{\"o}lder continuous. To demonstrate this we will follow the technique by \cite{gilbarg2015elliptic,serrin1964local, zamboni2002holder}. In what follows, for $1<p< \infty$, let us denote by $q$ its conjugate $(\frac{1}{p}+ \frac{1}{q}=1)$, and consider 
\begin{equation}
	L=-\left(y^{2\alpha_1}\partial_{x}^2+x^{2\alpha_2}\partial_{y}^2\right)+V.
\end{equation}
For any $s_0=(x_0,y_0)\in \Omega$, for convenience, write  $B_r$ for $B_{r}(s_0)$.

\begin{theorem}\label{202210292227}
	Let $u \in H_0(\Omega)$ be a weak solution of $Lu=f$ defined in $\Omega$ as in Definition \ref{202301241118},  where $f \in L^2(\Omega)$ and $V(x)\in L^p(\Omega)$, $1<p<\infty$, $V(x) \ge 0$. For any $s_0=(x_0,y_0)\in \Omega$ with $B_{4r}\subset \subset \Omega$,  then there is a positive $C$, we have
\begin{equation}
	\sup\limits_{B_r} |u| \le  C (\|u\|_{L^2(B_{2r})}+\|f\|_{L^2(\Omega)}).
\end{equation}
\end{theorem}

\begin{proof}
	Setting
	\begin{equation*}
		w=|u|+h,
	\end{equation*}
and 
\begin{equation}\label{202210281755}
	F(w)=
	\begin{cases}
		w^z, &h\le w \le l,\\
		zl^{z-1}w-(z-1)l^z,&l\le w,
	\end{cases}
\end{equation}
where $z\ge 1$, $l> h$, $h$	is a positive number that will be determined later.  Let us define the function 
\begin{equation}\label{202210281756}
	\begin{split}
	G(u)
	&={\rm sign}(u)\left(F(w)F'(w)-zh^{2z-1}\right)\\
	&=
	\begin{cases}
		{\rm sign}(u) \left(zw^{2z-1}-zh^{2z-1}\right), &h\leq w\leq l,\\
		{\rm sign}(u)\left(z^2l^{2z-2}w-z(z-1)l^{2z-1}-zh^{2z-1}\right), & w\geq l,
	\end{cases}
\end{split}
\end{equation}
it is easy to verify that
\begin{equation}\label{202210281757}
	G'=
	\begin{cases}
		(2-1/z)|F'|^2, &|u|< l-h,\\
		|F'|^2,        &|u|> l-h,
	\end{cases}
\end{equation}
and
	\begin{equation}\label{202210281758}
|G|\le FF', \quad wF'\le zF.
\end{equation}

For any open set $\Omega' \subset \subset \Omega$ and choose an open set $W$ such that $\Omega' \subset \subset W \subset \subset \Omega$. Consider the function $v=\eta^2G(u)$, where $\eta \in C_0^\infty(\mathbb{R}^2)$ and $0\le \eta \le 1$, $\eta \equiv 1$ on $\Omega'$, $\eta \equiv 0$ on $\mathbb{R}^2-W$. Since  $Dv=2\eta D\eta G(u)+ \eta^2 G'(u)Du$ for $|u| \neq l-h$, $Dv=2\eta D\eta G(u)$ and $Du=0$ a.e. \cite{serrin1964local} when $|u| = l-h$, thus $v \in H_0^1(\Omega)$ can be  obtained by \eqref{202210281755}, \eqref{202210281756}, \eqref{202210281757} and \eqref{202210281758}. Let $\Omega_1=\{x\in \Omega:|u|\le l-h\}$, $\Omega_2=\{x\in \Omega:|u|> l-h\}$, from \eqref{202210281758} we naturally know
\begin{equation*}
	\begin{split}
|G(u)| &\leq FF'\leq  zl^{2z-1} \	\mbox{in} \ \Omega_1,\\
|G(u)| &\le z^2l^{2z-2}|u|+z^2l^{2z-2}h-z(z-1)l^{2z-1} \	\mbox{in} \ \Omega_2.
	\end{split}	
\end{equation*}
Hence
\begin{equation*}
	\begin{split}
		\iint_{\Omega_1} Vv^2dxdy &\le (zl^{2z-1})^2\|V\|_{L^1(\Omega)}\le  (zl^{2z-1})^2\|V\|_{L^p(\Omega)},\\
		\iint_{\Omega_2} Vv^2dxdy&\le 2z^4l^{4z-4} \iint_{\Omega_2} V|u|^2dxdy+C(z,l,h)\|V\|_{L^1(\Omega)}\\
		&\le 2z^4l^{4z-4} \left\|V^\frac{1}{2}|u|\right\|_{L^2(\Omega)}+C(z,l,h)\|V\|_{L^p(\Omega)},
	\end{split}	
\end{equation*}
here $C(z,l,h)=2\left(z^2l^{2z-2}h-z(z-1)l^{2z-1}\right)^2$. These display that $v$ is a legitimate test function. Substituting $v$ in \eqref{202210282232} yields
\begin{equation}\label{202210282233}
\begin{split}
	&\iint_\Omega y^{2\alpha_1}\eta^2G'(u)\Big|\frac{\partial u}{\partial x}\Big|^2+x^{2\alpha_2}\eta^2G'(u)\Big|\frac{\partial u}{\partial y}\Big|^2+2y^{2\alpha_1}\eta G(u) \frac{\partial \eta}{\partial x} \frac{\partial u}{\partial x}+2x^{2\alpha_2}\eta G(u) \frac{\partial \eta}{\partial y} \frac{\partial u}{\partial y} dxdy\\
	&= \iint_\Omega \eta^2 f G(u)dxdy-\iint_\Omega\eta^2 VuG(u)dxdy.
\end{split}
\end{equation}
According to the definition of $\eta$ and \eqref{202210281757}, we have
\begin{equation}\label{202210290945}
	\begin{split}
		&\iint_\Omega y^{2\alpha_1}\eta^2G'(u)\Big|\frac{\partial u}{\partial x}\Big|^2+x^{2\alpha_2}\eta^2G'(u)\Big|\frac{\partial u}{\partial y}\Big|^2 dxdy\\
		&\ge \frac{1}{2}\min\limits_{W}\{x^{2\alpha_2},y^{2\alpha_1}\} \iint_{W} \eta^2G'(u)\Big|\frac{\partial u}{\partial x}\Big|^2+\eta^2G'(u)\Big|\frac{\partial u}{\partial y}\Big|^2 dxdy\\
		&=\frac{1}{2}\min\limits_{W}\{x^{2\alpha_2},y^{2\alpha_1}\} \iint_{\Omega} \eta^2G'(u)|Du|^2dxdy\\
		&\ge \frac{1}{2}\min\limits_{W}\{x^{2\alpha_2},y^{2\alpha_1}\} \iint_{\Omega} \eta^2 |F'|^2|Dw|^2dxdy.				
	\end{split}
\end{equation}
Let $\vartheta = \frac{1}{2}\min\limits_{W}\{x^{2\alpha_2},y^{2\alpha_1}\}$, utilizing \eqref{202210281758},  \eqref{202210282233},  \eqref{202210290945} and $DF(w)=F'Dw$, by the Young inequality we obtain 
\begin{equation*}
	\begin{split}
	&\iint_{\Omega} \eta^2 |DF(w)|^2 dxdy	\\ &\le 2\vartheta^{-1} \iint_\Omega\left( \eta|D\eta| |Dw||G(u)|+ \eta^2|f||G(u)|+ \eta^2 V |u||G(u)|   \right)dxdy  \\
	&\le 2\vartheta^{-1}\iint_\Omega \left(\eta |D\eta| F F' |Dw|+ zh^{-1} |f|\eta^2|F|^2+z \eta^2 V |F|^2\right)dxdy   \\
	&\le 2\vartheta^{-1} \iint_\Omega \left( \frac{\epsilon}{2} \eta^2  |F'|^2 |Dw|^2+\frac{1}{2\epsilon}|D\eta|^2|F|^2 + zh^{-1} |f|\eta^2|F|^2+z \eta^2 V|F|^2\right)dxdy,
	\end{split}
\end{equation*}
where the second inequality we have used  
\begin{equation*}
	|G(u)|\leq \frac{|u|+h}{h}|G(u)|=\frac{w}{h}|G(u)|\leq \frac{1}{h}FwF'\leq \frac{z}{h}|F|^2
\end{equation*}
and
\begin{equation*}
	|u||G(u)|\leq |u|FF'\leq FwF'\leq z|F|^2.
\end{equation*}
Selecting the  $\epsilon = \frac{\vartheta}{2}$, we know
\begin{equation}\label{202210291106}
		\iint_{\Omega} \eta^2 |DF(w)|^2 dxdy	 
		\le  C\iint_\Omega |D\eta|^2|F(w)|^2 + zh^{-1} |f|\eta^2|F(w)|^2+z \eta^2 |V||F(w)|^2  dxdy,
\end{equation}
here and in the following content in this proof, $C$ is independent on $z$. 
Since $\eta F(w) \in H_0^1(\Om)$, for any  $1\le 2X < \infty$, from Sobolev imbedding theorems and  inequality \eqref{202210291106} we obtain
 \begin{equation}\label{202210291127}
 		\begin{split}
 	\|\eta F(w)\|_{L^{2X}(\Omega)}^2 &\le C\|\eta F(w)\|_{H^1(\Omega)}^2\le C\iint_\Omega |D\eta|^2|F(w)|^2 +\eta^2 |DF(w)|^2 dxdy  \\
 	&\le C\iint_\Omega (z+1)|D\eta|^2|F(w)|^2 + zh^{-1} \eta^2|f||F(w)|^2+z \eta^2 |V||F(w)|^2 dxdy
        \end{split}
 \end{equation}
due to $z\geq 1$.  Moreover, we choose $\max\{2,q\}< X< \infty$, for any $\epsilon>0$ it follows that
\begin{equation*}
\begin{split}
\iint_\Omega \eta^2 |f||F|^2dxdy 
&\le \|f\|_{L^2(\Omega)} \|\eta F(w)\|_{L^4(\Omega)}^2\\
&\le \|f\|_{L^2(\Omega)} \left(\epsilon \|\eta F(w)\|_{L^{2X}(\Omega)}+ \epsilon^{-\mu_1}\|\eta F(w)\|_{L^2(\Omega)} \right)^2,\\
		\iint_\Omega \eta^2 V|F(w)|^2dxdy &\le \|V\|_{L^p(\Omega)} \|\eta F(w)\|_{L^{2q}(\Omega)}^2\\
		&\le \|V\|_{L^p(\Omega)} \left(\epsilon \|\eta F(w)\|_{L^{2X}(\Omega)}+ \epsilon^{-\mu_2}\|\eta F(w)\|_{L^2(\Omega)} \right)^2,
	\end{split}
\end{equation*}
where $\mu_1=\frac{X}{X-2}$, $\mu_2 =\frac{X(q-1)}{X-q}$.
Setting $h=\|f\|_{L^2(\Omega)}$ (if $f=0$, we may take $h>0$, eventually let $h \to 0$), and choosing the proper $\mu$ such that $\epsilon^{-2\mu} \ge \max\{\epsilon^{-2\mu_1}, \epsilon^{-2\mu_2}\} $, then the inequality \eqref{202210291127} becomes
\begin{equation}\label{202210291633}
		\|\eta F(w)\|_{L^{2X}(\Omega)}^2 
	\le C \iint_\Omega (z+1)|D\eta|^2|F(w)|^2      dxdy+ C\epsilon^2 z  \|\eta F(w)\|_{L^{2X}(\Omega)}^2 +C\epsilon^{-2\mu} z  \|\eta F(w)\|_{L^{2}(\Omega)}^2.
\end{equation}
 Now, choosing the appropriate $\epsilon$, in fact we may select $\epsilon^2= \frac{1}{2zC}$, thus
\begin{equation*}
	\begin{split}	
	\|\eta F(w)\|_{L^{2X}(\Omega)}^2 
	&\le C \iint_\Omega (z+1)|D\eta|^2|F(w)|^2 dxdy  +Cz^{\mu+1} \|\eta F(w)\|_{L^{2}(\Omega)}^2\\
	&\le C(z+1)^{2(\mu+1)} \iint_\Omega (|D\eta|+\eta)^2   |F(w)|^2dxdy,
\end{split}
\end{equation*}
that is, 
\begin{equation}\label{202302171218}	
		\|\eta F(w)\|_{L^{2X}(\Omega)}
		\le C(z+1)^{\mu+1}   \|(|D\eta|+\eta)F(w)\|_{L^2(\Omega)}.
\end{equation}

It is now desirable to specify the cut-off function $\eta$ more precisely. Let $r_1$ and $r_2$ be such that $r\le r_1 < r_2 \le 3r$, selecting $\eta$ in such a way that $0 \le \eta \le 1$, $\eta\big|_{B_{r_1}}=1$ and $\eta \big|_{\mathbb{R}^2-B_{r_2}}=0$, $|D\eta|\le \frac{2}{r_2-r_1}$. By \eqref{202302171218},
\begin{equation}\label{202210291805}
	\| F(w)\|_{L^{2X}(B_{r_1})}
	\le C \frac{(z+1)^{\mu+1}}{r_2-r_1}  \|F\|_{L^2(B_{r_2})}.
\end{equation}
Setting $\Phi(p,r)=\left(\iint_{B_r}|w|^pdxdy\right)^\frac{1}{p}$, letting $l \to \infty$, by \eqref{202210291805}, then  
\begin{equation*}
	\Phi(2zX,r_1) \le \left( C \frac{(z+1)^{\mu+1}}{r_2-r_1}  \right)^\frac{1}{z}\Phi(2z,r_2)\le \left( C \frac{(2z)^{\mu+1}}{r_2-r_1}  \right)^\frac{1}{z}\Phi(2z,r_2).
\end{equation*}
Taking $\theta=2z$, then the above inequality is transformed into
\begin{equation}\label{202210292045}
	\Phi(\theta X,r_1) \le  \left( C \frac{\theta^{\mu+1}}{r_2-r_1}  \right)^\frac{2}{\theta}\Phi(\theta,r_2).
\end{equation}
This inequality can now be iterated to receive the desired results. Indeed, we may choose $\theta=\theta^m=2X^m$, and $r^m=r+\frac{r}{2^m}$, $m=0,1,2 \cdots$, the iteration yields 
\begin{equation}\label{202210292111}
	\Phi(2X^m,r^m)\le (CX)^{2(\mu+1) \Sigma mX^{-m}}\Phi(p,2r)\le C\Phi(2,2r).
\end{equation} 
Consequently, letting $m \to \infty$, we obtain
\begin{equation}\label{202210292116}
	\sup\limits_{B_r}w \le  C \|w\|_{L^2(B_{2r})}.
\end{equation}
Following the definition of $w$, we can see that
\begin{equation}\label{202210292120}
	\sup\limits_{B_r} |u| \le  C (\|u\|_{L^2(B_{2r})}+\|f\|_{L^2(\Omega)}).
\end{equation}
\end{proof}

\begin{theorem}\label{202210292207}
	Let $u \ge 0$, $u \in H_0(\Omega)$ be a weak solution of $Lu-cu=0$ defined in $\Omega$ as in Definition \ref{202301241118},
  $s_0=(x_0,y_0) \in \Omega$, where $f \in L^2(\Omega)$ and $V(x)\in L^p(\Omega)$, $1<p<\infty$, $V(x) \ge 0$, $c$ is a constant. Then for any $B_{4r} \subset \subset \Omega$,  there is a positive $C$ such that
	\begin{equation}
		\sup\limits_{B_r} u \le  C \inf\limits_{B_r} u .
	\end{equation}
\end{theorem}	

\begin{proof}
From the Section 2 we know that $u \in H_{\rm loc}^1(\Omega)$ and $\iint_\Omega Vu^2dx< \infty$, and $u \in L^\infty (B_{4r})$ by Theorem \ref{202210292227}.	Define 
\begin{equation*}
	\bar{u}=u+k, 
\end{equation*}
where we may choose arbitrary $k>0$ and eventually let $k \to 0$.   Consider the function $v=\eta^2\bar{u}^\beta$, where $\eta \in C_0^\infty(\mathbb{R}^2)$ and $0\le \eta \le 1$, $\supp \eta\subset B_{3r}(s_0)$ and $\beta\in\mathbb{R}$. Through simple verification, we can get $v$ is a valid test function.

Since $Dv=2\eta D\eta \bar{u}^\beta+ \beta \eta^2 \bar{u}^{\beta-1}Du$, by substitution into \eqref{202210282232} for $\beta \neq 0$, we find that
\begin{equation*}
	\begin{split}
		&\iint_{\Omega} \eta^2 y^{2\alpha_1} \bar{u}^{\beta-1} \left|\frac{\pt{u}}{\pt{x}}\right|^2+\eta^2 x^{2\alpha_2} \bar{u}^{\beta-1} \left|\frac{\pt{u}}{\pt{y}}\right|^2dxdy	\\
		&= \frac{1}{\beta}\iint_{\Omega} cuv-Vuvdxdy- \frac{2}{\beta} \iint_{\Omega} \eta y^{2\alpha_1} \frac{\pt{\eta}}{\pt{x}}  \frac{\pt{u}}{\pt{x}} \bar{u}^\beta-\eta x^{2\alpha_2} \frac{\pt{\eta}}{\pt{y}}  \frac{\pt{u}}{\pt{y}} \bar{u}^\beta dxdy\\
		&\le  \frac{1}{|\beta|} \iint_{\Omega} (|c|+ V) \eta^2 \bar{u}^{\beta+1}dxdy
		+ \frac{\epsilon}{|\beta|}\iint_{\Omega} \eta^2 y^{2\alpha_1} \left|\frac{\pt{u}}{\pt{x}} \right|^2 \bar{u}^{\beta-1}
		+\eta^2 x^{2\alpha_2} \left|\frac{\pt{u}}{\pt{y}} \right|^2 \bar{u}^{\beta-1}dxdy\\
		&\hspace{6mm}+ \frac{1}{|\beta|\epsilon}\iint_{\Omega} y^{2\alpha_1} \left|\frac{\pt{\eta}}{\pt{x}} \right|^2 \bar{u}^{\beta+1}+x^{2\alpha_2}\left|\frac{\pt{\eta}}{\pt{y}} \right|^2 \bar{u}^{\beta+1}dxdy.
	\end{split}
\end{equation*}
Selecting $\epsilon =\min\{1,\frac{|\beta|}{4}\}$ and utilizing the feature of $\eta$, we have
\begin{equation}\label{202210301844}
	\begin{split}
		&\iint_{\Omega} \eta^2 y^{2\alpha_1} \bar{u}^{\beta-1} \left|\frac{\pt{u}}{\pt{x}}\right|^2+\eta^2 x^{2\alpha_2} \bar{u}^{\beta-1} \left|\frac{\pt{u}}{\pt{y}}\right|^2dxdy	\\
		&\le  C(|\beta|) \iint_{\Omega} |D\eta|^2\bar{u}^{\beta+1}+ (|c|+V)\eta^2 \bar{u}^{\beta+1}dxdy,	
	\end{split}
\end{equation}
and as in   \eqref{202210290945} we find that
\begin{equation}\label{202210302141}
	\iint_{\Omega} \eta^2 y^{2\alpha_1} \bar{u}^{\beta-1} \left|\frac{\pt{u}}{\pt{x}}\right|^2+\eta^2 x^{2\alpha_2} \bar{u}^{\beta-1} \left|\frac{\pt{u}}{\pt{y}}\right|^2dxdy \ge C \iint_{\Omega} \eta^2 \bar{u}^{\beta-1}|Du|^2dxdy.
\end{equation}
Combining \eqref{202210301844} and \eqref{202210302141} yields
\begin{equation}\label{202210302146}
	\iint_{\Omega} \eta^2 \bar{u}^{\beta-1}|Du|^2dxdy \le C(|\beta|) \iint_{\Omega} |D\eta|^2\bar{u}^{\beta+1}+ (|c|+ V) \eta^2 \bar{u}^{\beta+1}dxdy.
\end{equation}
To proceed further, setting
\begin{equation*}
	\Theta=
	\begin{cases}
		\bar{u}^\frac{\beta+1}{2}, &\beta \neq -1,\\
		\log{\bar{u}}, &\beta =-1,
	\end{cases}
\end{equation*}
we may rewrite \eqref{202210302146} by
\begin{equation}\label{202211011703}
	\begin{cases}
		\iint_{\Omega}|\eta D\Theta|^2dxdy \le C(|\beta|)(\beta+1)^2 \iint_{\Omega}\left( |D\eta|^2\bar{u}^{\beta+1}+(|c|+ V)\eta^2 \bar{u}^{\beta+1}\right)dxdy, &\beta \neq -1,\\
		\iint_{\Omega} |\eta D\Theta|^2dxdy \le C\iint_{\Omega}\left( |D\eta|^2+(|c|+ V)\eta^2\right) dxdy, &\beta =-1.
	\end{cases}
\end{equation}

We may apply Sobolev  imbedding theorems to obtain for any  $1\le 2X < \infty$
\begin{equation}\label{202210302209}
	\|\eta \Theta\|_{L^{2X}(\Omega)}^2 \le C\|\eta \Theta\|_{H_0^1(\Omega)}^2
	\le C\iint_{\Omega} \left(|D\eta|^2|\Theta|^2 +|\eta D\Theta|^2\right)dxdy.  		
\end{equation}
Consider $q< X <\infty$, we obtain
\begin{equation*}
	\begin{split}
		\iint_{\Omega} \eta^2 (|c|+ V)|\Theta|^2dxdy 
		&\le \|(|c|+ V)\|_{L^p(\Omega)} \|\eta \Theta\|_{L^{2q}(\Omega)}^2\\
		&\le \|(|c|+ V)\|_{L^p(\Omega)} \left(\epsilon \|\eta \Theta\|_{L^{2X}(\Omega)}+ \epsilon^{-\mu}\|\eta \Theta\|_{L^2(\Omega)} \right)^2,
	\end{split}
\end{equation*}
where  $\mu =\frac{X(q-1)}{X-q}$. Letting  $\gamma=\beta +1$, by \eqref{202210302209}  it then follows that 
\begin{equation*}
	\|\eta\Theta\|_{L^{2X}(\Omega)}^2 \le C\left((\gamma^2+1)\|D\eta\Theta\|_{L^2(\Omega)}^2+\gamma^2\epsilon^2\|\eta\Theta\|_{L^{2X}(\Omega)}^2+\gamma^2 \epsilon^{-2\mu} \|\eta \Theta\|_{L^2(\Omega)}^2      \right),
\end{equation*}
where $C=C(|\beta|)$ is bounded when $|\beta|$ is bounded away from zero.  Furthermore, choosing a suitable $\epsilon$, we then get
\begin{equation}\label{202211011041}
	\|\eta\Theta\|_{L^{2X}(\Omega)} \le C(|\gamma|+1)^{\mu+1}\|(D\eta+\eta)\Theta\|_{L^2(\Omega)}.
\end{equation}

Now, the more accurate cut-off function $\eta$ will be given, let $r \le r_1 <r_2 \le 3r$ such that $\eta|_{B_{r_1}(s_0)}\equiv 1$, $\eta|_{\mathbb{R}-B_{r_2}(s_0)}\equiv 0$ with $|D\eta|\le \frac{2}{r_2-r_1}$. Hence,  in the light of \eqref{202211011041}, 
\begin{equation}\label{202211011044}
	\|\Theta\|_{L^{2X}(B_{r_1}(s_0))} \le C \frac{(|\gamma|+1)^{\mu+1}}{r_2-r_1} \|\Theta\|_{L^2(B_{r_2}(s_0))}.
\end{equation}

For any $B_{4r}(s_0) \subset \Omega$ and $p \neq 0$,  now we introduce 
\begin{equation*}
	\Psi(p,r)=\left(\iint_{B_r(s_0)}|\bar{u}|^pdxdy  \right)^\frac{1}{p},
\end{equation*}
in fact,
\begin{equation*}
	\begin{split}
		\Psi(\infty,r)&= \lim\limits_{p \to \infty} \Psi(p,r)=\sup\limits_{B_r(s_0)} \bar{u},\\
		\Psi(-\infty,r)&= \lim\limits_{p \to -\infty} \Psi(p,r)=\inf\limits_{B_r(s_0)} \bar{u}.
	\end{split}
\end{equation*}	
From \eqref{202211011044},
\begin{equation}\label{202211011156}
	\begin{cases}
		\Psi(\gamma X,r_1) \le \left( \frac{(C|\gamma|+1)^{\mu+1}}{r_2-r_1} \right)^{2/|\gamma|}\Psi(\gamma,r_2),&\gamma>0,\\[1mm]
		\Psi(\gamma ,r_2) \le \left( \frac{(C|\gamma|+1)^{\mu+1}}{r_2-r_1} \right)^{2/|\gamma|}\Psi(\gamma X,r_1),&\gamma<0.		
	\end{cases}
\end{equation}
When $\beta>0$, we have $\gamma>1$, taking $p>1$, $\gamma=\gamma^m=X^mp$ and  $r^m=r+\frac{r}{2^m}$, $m=0,1,\cdots$, consequently, by inequality \eqref{202211011156},
\begin{equation*}
	\Psi(X^mp,r^m)\le (CX)^{2(\mu+1) \Sigma mX^{-m}}\Psi(p,2r)\le C\Psi(p,2r),
\end{equation*}    
letting $m \to \infty$, 
\begin{equation}\label{202211031642}
	\sup_{B_r}\bar{u}\le C\|\bar{u}\|_{L^p(B_{2r})}.
\end{equation}

For $\beta<0$, then $\gamma<1$, in a similar manner, we can prove that for any $p_0, p$ such that $0<p_0<p<X$,
\begin{equation}\label{202211011641}
	\begin{cases}
		\Psi(p,2r) \le C\Psi(p_0,3r),\\
		\Psi(-p_0,3r) \le C\Psi(-\infty,r).
	\end{cases}
\end{equation}
In reality, we only require to  prove that 
\begin{equation}\label{202211011652}
	\Psi(p_0,3r) \le C\Psi(-p_0,3r)
\end{equation}
for some $p_0$. Now, to verify \eqref{202211011652}, we turn to the second estimate of \eqref{202211011703}.

Choosing $\eta$ in such a way that $\eta |_{B_d}(s_0) \equiv 1$, $\supp \eta \subset B_{2d}(s_0) \subset B_{4r}(s_0)$, and $|D\eta| \le C$, where $B_{d}(s_0)$ is an arbitrary open ball contained in $B_{2r}(s_0)$. Thanks to H\"older inequality,  from \eqref{202211011703} we then have
\begin{equation*}
	\iint_{B_d(z_0)} |D\Theta|dxdy \le Cd,
\end{equation*}
by invoking the Theorem 7.21 \cite{gilbarg2015elliptic}, there exists a constant $p_0>0$ such that for
\begin{equation*}
	\Theta_0=\frac{1}{|B_{3r(s_0)}|} \iint_{B_{3r(s_0)}} \Theta dxdy,
\end{equation*}
we have
\begin{equation*}
	\iint_{B_{3r}(s_0)} e^{p_0 |\Theta-\Theta_0|} dxdy \le C.
\end{equation*}
Linking with the definition of $\Theta$, the  desired result \eqref{202211011652} is obtained. Combining \eqref{202211011641} and \eqref{202211011652}, we know that
\begin{equation}\label{202212052235}
	\|\bar{u}\|_{L^p(B_{2r}(s_0))} \le C \inf_{B_r(s_0)} \bar{u}.
\end{equation}
Now, recall the inequalities \eqref{202211031642} and \eqref{202212052235}, that imply
\begin{equation}\label{202212052240}
	\sup_{B_r} \bar{u} \le C \inf_{B_r} \bar{u}.
\end{equation}
\end{proof}

\begin{lemma}\label{202301111641}
  Let $s=(x,y)$, $D_i^lu(s)=\frac{u(s+le_i)-u(s)}{l} \ (i=1,2)$ denote the $i\raisebox{0mm}{-}th$ difference quotient of size $l$ for $s\in \Omega'$, $l\in \mathbb{R}$, $0<|l|<\dist(\Omega',\partial{\Omega})$, $\Omega' \subset\subset \Omega$. Suppose $u\in \mcH_0^1(\Omega)$, then for any $\Omega' \subset\subset \Omega$, we have 
	\begin{equation*}
		\|D^lu\|_{L^2(\Omega')}^2	\le \iint_\Om y^{2\alpha_1}|\partial_{x}u|^2+ x^{2\alpha_2}|\partial_{y}u|^2dxdy \le 
		C\|u\|_{\mcH_0^1(\Omega)}^2
	\end{equation*}	
	for some constant $C$ and all $0<|l|<\frac{1}{2}\dist(\Omega',\partial{\Omega})$, where $D^lu=(D_1^lu,D_2^lu)$.
\end{lemma}

\begin{proof}
 Suppose $u\in C_0^\infty(\Omega)$, for any $s\in \Omega'$, $i=1,2$, $0<|l|<\frac{1}{2}\dist(\Omega',\partial{\Omega})$, we see that
	\begin{equation*}
		u(x+l,y)-u(x,y)=l\int_0^1 \partial_{x}u(x+tl,y)dt,
	\end{equation*}
	so that 
	\begin{equation*}
		|u(x+l,y)-u(x,y)|\le |l|\int_0^1 |\partial_{x}u(x+l,y)|dt, 
	\end{equation*}
	i.e.
	\begin{equation*}
		|D_1^lu(x,y)|\le \int_0^1 |\partial_{x}u(x+tl,y)|dt.	
	\end{equation*}
Similarly, we have $|D_2^lu(x,y)|\le \int_0^1 |\partial_{y}u(x,y+tl)|dt$. Hence, for any 
$\Omega' \subset\subset \Omega$, by Cauchy inequality, we obtain that 
	\begin{equation*}
		\begin{split}
			c_1\iint_{\Omega'} |D^lu(x)|^2dxdy 
			&\le  \iint_{\Om'}\left( y^{2\alpha_1}|D^l_1u(x)|^2+x^{2\alpha_2}|D^l_2u(x)|^2\right)dxdy \\
			& \leq \iint_{\Om'}y^{2\alpha_1}\left(\int_0^1|\partial_{x}u(x+tl,y)|dt \right)^2+x^{2\alpha_2}\left(\int_0^1|\partial_{y}u(x,y+tl)|dt \right)^2 dxdy  \\			
			&\le  \int_0^1\iint_{\Om'}\left(y^{2\alpha_1} |\partial_{x}u(x_1+tl,y)|^2+x^{2\alpha_2}|\partial_{y}u(x, y+tl)|^2 \right)dxdydt\\
			&\le  \iint_\Om y^{2\alpha_1}|\partial_{x}u(x,y)|^2+ x^{2\alpha_2}|\partial_{y}u(x,y)|^2dxdy,
		\end{split}
	\end{equation*}
where $c_1=\frac{1}{2}\min\limits_{\Omega'}\{x^{2\alpha_2},y^{2\alpha_1}\}>0$, so that 
\begin{equation*}
\iint_{\Omega'} |D^lu(x)|^2dxdy \le C \left( \iint_\Om y^{2\alpha_1}|\partial_{x}u|^2+ x^{2\alpha_2}|\partial_{y}u|^2dxdy \right).
\end{equation*}
Since $C_0^\infty(\Omega)$ is dense in $\mcH_0^1(\Omega)$, then the above inequality is established for any $u\in \mcH_0^1(\Omega)$.
\end{proof}

\begin{theorem}\label{201301111756}
	Suppose that $u\in H_0(\Omega)$ is the weak solution for the problem \eqref{202301231543} on $\Omega$, where $V \in L^p(\Omega)$, $2 \le p<\infty$, $f \in L^2(\Omega)$, then $u\in H_{\rm loc}^2(\Omega)$. Furthermore,
	$u \in C^{0,\alpha}(\overline{\Omega'})$, where $\alpha \in (0,1)$, $\Omega'\subset\subset \Omega$.
\end{theorem}

\begin{proof}
	1. For any subset $\Omega'\subset\subset \Omega$, we may choose an open set $W$ such that $\Omega' \subset \subset W \subset \subset \Omega$. In addition, we select a  function $ \eta \in C_0^\infty(\mathbb{R}^2)$	such that $\eta \equiv1$ on $\Omega'$, $\eta \equiv 0$ on $\mathbb{R}^2- W$ and $0\le\eta \le 1$.
	Note that $u$ is the weak solution for $Lu=f$, then for any $v\in \mcH_0^1(\Omega)$:
	\begin{equation}\label{Th.2.7.1}
	\iint_{\Omega} y^{2\alpha_1}\pt_{x}u \pt_{x}v+x^{2\alpha_2}\pt_{y}u \pt_{y}vdxdy=\iint_{\Omega} (fv-Vuv)dxdy.		
	\end{equation}
 Let 
	\begin{equation*}
		A=	\iint_{\Omega} y^{2\alpha_1}\pt_{x}u \pt_{x}v+x^{2\alpha_2}\pt_{y}u \pt_{y}vdxdy,\quad B=\iint_{\Omega} (fv-Vuv)dxdy, 
	\end{equation*}
	then $A=B$.
	
	2. Let $0<|l|<\frac{1}{2} \min\{\dist(\Omega',\pt W),\dist(W,\pt \Omega)\} $ and consider that $|l|$ be sufficiently small, then substitute $v=-D_k^{-l}(\eta^2D_k^lu)$ into \eqref{Th.2.7.1}, $k\in \{1,2\}$. 
	Indeed, 
	\begin{equation*}
		\begin{split}
			v&=-\frac{1}{l}D_k^{-l}\left(\eta^2(x)[u(x+le_k)-u(x)]\right)\\
			&=\frac{1}{l^2}\left(\eta^2(x-le_k)[u(x)-u(x-le_k)] -\eta^2(x)[u(x+le_k)-u(x)]\right),	
		\end{split}	
	\end{equation*}
	since $u\in H_{\rm loc}^1(\Omega)$, $V^{\frac{1}{2}}u \in L^2(\Omega)$ and $\supp \eta\subset W$, thus $v\in H_0^1(\Omega) \subset  \mcH_0^1(\Omega)$, and $V^{\frac{1}{2}}v \in L^2(\Omega)$.	
	Then
	\begin{align*}\label{Th.2.7.2}
			A&=	- \left( \iint_{\Omega} y^{2\alpha_1}\pt_{x}u \pt_{x}\left(D_k^{-l}(\eta^2D_k^lu)\right)+ x^{2\alpha_2}\pt_{y}u \pt_{y}\left(D_k^{-l}(\eta^2D_k^lu) \right) dxdy\right)   \\
			&=\iint_{\Omega} D_k^l(y^{2\alpha_1}\pt_{x}u)	\pt_{x}(\eta^2D_k^lu)+ D_k^l(x^{2\alpha_2}\pt_{y}u)	\pt_{y}(\eta^2D_k^lu) dxdy \\
			&=\iint_{\Omega}  (y^{2\alpha_1})_k^l(D_k^l\pt_{x}u)\pt_{x}(\eta^2D_k^lu)+(D_k^ly^{2\alpha_1})\pt_{x}u\pt_{x}(\eta^2D_k^lu)\\
            &\quad+(x^{2\alpha_2})_k^l(D_k^l\pt_{y}u)\pt_{y}(\eta^2D_k^lu)+(D_k^lx^{2\alpha_2})\pt_{y}u\pt_{y}(\eta^2D_k^lu)dxdy\\	
			&=\iint_{\Omega}(y^{2\alpha_1})_k^l(D_k^l\pt_{x}u)(D_k^l\pt_{x}u)\eta^2+(x^{2\alpha_2})_k^l(D_k^l\pt_{y}u)(D_k^l\pt_{y}u) \eta^2dxdy\\
			&\quad  + \iint_{\Omega} 2\eta (y^{2\alpha_1})_k^l (\pt_{x}\eta)(D_k^l\pt_{x}u)(D_k^lu)+ 2\eta (\pt_{x}\eta)(D_k^ly^{2\alpha_1})\pt_{x}u(D_k^lu)+\eta^2 (D_k^ly^{2\alpha_1})\pt_{x}u\pt_{x}(D_k^lu)			
			\\
			&\quad +2\eta (x^{2\alpha_2})_k^l (\pt_{y}\eta)(D_k^l\pt_{y}u)(D_k^lu)+ 2\eta (\pt_{y}\eta)(D_k^lx^{2\alpha_2})\pt_{y}u(D_k^lu)+\eta^2 (D_k^lx^{2\alpha_2})\pt_{y}u\pt_{y}(D_k^lu)dxdy\\			
			&=A_1+A_2,
	\end{align*}
	where 
	\begin{equation*}
	(y^{2\alpha_1})_k^l=
		\begin{cases}
			y^{2\alpha_1},&k=1,\\
			(y+l)^{2\alpha_1},&k=2,
		\end{cases}
	\end{equation*}
	
\begin{equation*}
	(x^{2\alpha_2})_k^l=
	\begin{cases}
		(x+l)^{2\alpha_2},&k=1,\\
		x^{2\alpha_2},  &k=2.
	\end{cases}
\end{equation*}	
 According to the definition of $\eta$, we have
	\begin{equation}\label{Th.2.7.3}
		A_1\ge \delta \iint_{\Omega} \eta^2 |D_k^lDu|^2dxdy \ge \theta  \iint_{\Omega} \eta^2 |D_k^lDu|^2dxdy
	\end{equation}
	for some proper constant $\theta$, $\delta \in (0,1)$, and given that $y^{2\alpha_1},x^{2\alpha_2} \in C^1(0,1]$
	\begin{equation*}
		|A_2|\le C \iint_{\Omega}\eta |D_k^lDu||D_k^lu|+\eta|D_k^lDu||Du|+\eta|D_k^lu||Du|dxdy. 
	\end{equation*}
	Furthermore, by  $\supp \eta\subset \overline{W}$ and Cauchy's inequality with $\epsilon$, then
	\begin{equation*}
		|A_2|\le \epsilon \iint_{\Omega} \eta^2|D_k^lDu|^2dx+\frac{C}{\epsilon}\iint_{W} \left(|D_k^lu|^2+|Du|^2\right)dxdy.
	\end{equation*}
	By invoking the result of Lemma  \ref{202301111641}, we see that $\iint_{W} |D_k^lu|^2dxdy \le C\|u\|_{\mcH_0^1(\Omega)}^2$. Moreover,
	\begin{equation*}
	 \iint_{W} |Du|^2dxdy\le
	\big( \min\limits_{W}\{x^{2\alpha_2},y^{2\alpha_1}\} \big)^{-1}\|u\|_{\mcH_0^1(\Omega)}^2  \le C\|u\|_{\mcH_0^1(\Omega)}^2. 	
	\end{equation*}
 We may choose $\epsilon=\frac{\theta}{2}$, hence
	\begin{equation}\label{Th.2.7.4}
		|A_2|\le \frac{\theta}{2} \iint_{\Omega} \eta^2 |D_k^lDu|^2dxdy + C \|u\|_{\mcH_0^1(\Omega)}^2.
	\end{equation}
	Combining \eqref{Th.2.7.3} with \eqref{Th.2.7.4}, we obtain
	\begin{equation}\label{Th.2.7.5} 
		A \ge \frac{\theta}{2} \iint_{\Omega} \eta^2 |D_k^lDu|^2dxdy -C\|u||_{\mcH_0^1(\Omega)}^2.	
	\end{equation}

	3. Since $v=-D_k^{-l}(\eta^2D_k^lu)$, by Lemma \ref{202301111641} we know 
	\begin{equation*}
		\begin{split}
			\iint_{\Omega} |v|^2dx &\le  C \iint_{\Omega}  y^{2\alpha_1}|\partial_{x}(\eta^2D_k^lu  )|^2+x^{2\alpha_2}|\partial_{y}(\eta^2D_k^lu  )|^2dxdy \\                                 			
			 &\le C\iint_{W} |D_k^lu|^2+\eta^2|D_k^lDu|^2 dxdy\\
			&\le  C\|u||_{\mcH_0^1(\Omega)}^2+C \iint_{\Omega} \eta^2|D_k^lDu|^2dxdy.
		\end{split}
	\end{equation*}
	Applying the Cauchy's inequality with $\epsilon$ and Theorem \ref{202210292227}, and based on the fact $V \in L^p(\Omega)$, $2\le p<\infty$,
	\begin{equation*}
		\begin{split}		
			|B|&\le \iint_{\Omega} \left(|f||v|+V|u||v|\right)dxdy\\
			& \le \|f\|_{L^2(\Omega)} \|v\|_{L^2(\Omega)}+\|u\|_{L^\infty(W)} \iint_{\Omega} V|v|dxdy\\
			& \le \|f\|_{L^2(\Omega)} \|v\|_{L^2(\Omega)}+C(\|f\|_{L^2(\Omega)}+ \|u\|_{L^2(\Omega)})     \iint_{\Omega} V|v|dxdy\\	
			& \le \|f\|_{L^2(\Omega)} \|v\|_{L^2(\Omega)}+C(\|f\|_{L^2(\Omega)}+ \|u\|_{L^2(\Omega)})  \|V\|_{L^2(\Omega)} \|v\|_{L^2(\Omega)}	\\
			& \le  \|f\|_{L^2(\Omega)} \|v\|_{L^2(\Omega)}+C(\|f\|_{L^2(\Omega)}+ \|u\|_{L^2(\Omega)})  \|V\|_{L^p(\Omega)} \|v\|_{L^2(\Omega)}\\
			& \le  C(\|f\|_{L^2(\Omega)}+ \|u\|_{L^2(\Omega)}) \|v\|_{L^2(\Omega)}\\				
			&\le \epsilon \iint_{\Omega}
			 \eta^2|D_k^lDu|^2dxdy+\frac{C}{\epsilon}\iint_{\Omega}|f|^2+|u|^2dxdy+\frac{C}{\epsilon}\|u||_{\mcH_0^1(\Omega)}^2.
		\end{split}	
	\end{equation*}
	Similarly, we choose $\epsilon=\frac{\theta}{4}$, then
	\begin{equation}\label{Th.2.7.6}
		|B|\le 	\frac{\theta}{4} \iint_{\Omega} \eta^2|D_k^lDu|^2dxdy+C\iint_{\Omega}|f|^2+|u|^2dxdy+C\|u||_{\mcH_0^1(\Omega)}^2.
	\end{equation}
	
	4. Thanks to \eqref{Th.2.7.5} and \eqref{Th.2.7.6}, we observe that for $k=1,2$, 
	\begin{equation*}
		\frac{\theta}{4} \iint_{\Omega'} \eta^2|D_k^lDu|^2dxdy\le C\iint_{\Omega} |f|^2+|u|^2 dxdy+C\|u||_{\mcH_0^1(\Omega)}^2
	\end{equation*}
 for any sufficiently small $|l| \neq 0$. Furthermore, since $u \in H^1(\Omega')$, utilizing the result $(2)$ of Theorem 3 \cite[Chapter 5.8.2]{EvansPartial}, we know $Du\in H^1(\Omega')$, hence $u\in H_{\rm loc}^2(\Omega)$.  By the classical Sobolev compact embedding theorem, we know $u \in C^{0,\alpha}(\overline{\Omega'})$, where $\alpha \in (0,1)$.
\end{proof}

\section{ The first eigenvalue}

In this section, we are interested in extremum problems involving the first eigenvalue of problem \eqref{202301172048} when $V \in S_p$, $1<p<\infty$. In particular, we first discuss the properties of the spectrum, which paves the way for finding the optimal potential function. Finally, we characterize the optimal potential function and prove its uniqueness,  where some of ideas were developed in 
\cite{notarantonio1998extremal,egnell1987extremal}. In order to analyze the properties of the spectrum, we rely on the following lemmas, the similar results may be displayed on \cite{EvansPartial,gilbarg2015elliptic,NonlinearPotential}, for the sake of clarity, we will provide specific proof.

\begin{lemma}\label{202301080919}
	Suppose $F: \mathbb{R} \to \mathbb{R}$ is $C^1$ with $F'$ is bounded, then for any $u \in \mcH_0^1(\Omega)$, we have 
			\begin{equation}\label{202201081124}
				\begin{split}
					&F(u) \in \mcH^1(\Omega),\\
					y^{\alpha_1}\partial_{x}F(u)=y^{\alpha_1}F'(u)&\partial_{x}u,x^{\alpha_2}\partial_{y}F(u)=x^{\alpha_2}F'(u)\partial_{y}u.
				\end{split}
			\end{equation}		
		\end{lemma}
		\begin{proof}
			Firstly, we will show that
			\begin{equation}
				\partial_{x}(y^{\alpha_1}F(u))=y^{\alpha_1}F'(u)\partial_{x}u.
			\end{equation}
			In fact, for any $u \in \mcH_0^1(\Omega)$, there exists a sequence $(u_n)_{n \in \mathbb{N}} \subset C_0^\infty(\Omega)$	such that
			\begin{equation}\label{202301081032}
				\|u_n-u\|_{\mcH^1(\Omega)} \to 0, 	
			\end{equation}
			up to a subsequence, we know 
			\begin{equation}\label{202301081034}
				u_n \to u \ a.e.,	
			\end{equation}
			since $F'$ is continuous, 
			\begin{equation}\label{202301081035}
				F'(u_n) \to F'(u)  \  a.e..	
			\end{equation}
Moreover, $|F(u_n)-F(0)|\le \|F'\|_{L^\infty(\Omega)} |u_n|\in L^2(\Omega)$, then
\begin{equation*}
	|F(u_n)|\le |F(0)|+\|F'\|_{L^\infty(\Omega)} |u_n|,
\end{equation*} 		
i.e. $F(u_n) \in L^2(\Omega)$.  
Note that
\begin{equation*}
	\iint_\Omega |F(u_n)-F(u)|^2dxdy \le \iint_\Omega \|F'\|^2_{L^\infty(\Omega)} |u_n-u|^2dxdy \to 0, \ n\to \infty.
\end{equation*}
For any $\phi \in C_0^\infty(\Omega)$,
 			\begin{equation}\label{202301081110}
				\begin{split}
					&\iint_{\Omega} y^{\alpha_1}F(u)\partial_{x}\phi dxdy=\lim\limits_{n \to \infty}  \iint_{\Omega} y^{\alpha_1}F(u_n)\partial_{x} \phi dxdy\\				
					&= -\lim\limits_{n \to \infty} \iint_{\Omega} \partial_{x} (y^{\alpha_1}F(u_n)) \phi dxdy
					=-\lim\limits_{n \to \infty} \iint_{\Omega} y^{\alpha_1}  \partial_{x} F(u_n) \phi dxdy\\
					&=-\lim\limits_{n \to \infty} \iint_{\Omega}y^{\alpha_1} F'(u_n)\partial_{x} u_n \phi dxdy=
					-\iint_{\Omega} y^{\alpha_1} F'(u)\partial_{x} u \phi dxdy,
				\end{split}
			\end{equation}
			indeed, for the last equality, applying the fact \eqref{202301081032} \eqref{202301081034} \eqref{202301081035}, and given that $F \in C^1$ and $F'$ is bounded, using the dominated convergence theorem, 
			\begin{equation*}
				\begin{split}
					&\iint_{\Omega} \big|\left(y^{\alpha_1} F'(u)\partial_{x} u- y^{\alpha_1} F'(u_n)\partial_{x} u_n      \right) \phi \big|dxdy\\
					&\le \iint_{\Omega} |F'(u_n)-F'(u)| |y^{\alpha_1}\partial_{x}u| |\phi|+ |F'(u_n)| |y^{\alpha_1}\partial_{x}u_n-y^{\alpha_1}\partial_{x}u| |\phi|dxdy	\to 0, \ n \to \infty.
				\end{split}
			\end{equation*}
			Moreover, for any $\phi \in C_0^\infty(\Omega)$,
			\begin{equation}\label{202301081122}
				\begin{split}
					\iint_\Omega (y^{\alpha_1} \partial_{x}F(u)) \phi dxdy& =\iint_{\Omega} \partial_{x}F(u)(y^{\alpha_1}\phi)dxdy=-\iint_{\Omega} F(u) (\partial_{x}(y^{\alpha_1}\phi))dxdy\\
					&=-\iint_{\Omega} (y^{\alpha_1}F(u)) \partial_{x}\phi dxdy=\iint_{\Omega} \partial_{x}(y^{\alpha_1}F(u)) \phi dxdy.
				\end{split}
			\end{equation}
			Combining \eqref{202301081110} and \eqref{202301081122}, we see that $y^{\alpha_1}\partial_{x}F(u)=y^{\alpha_1}F'(u)\partial_{x}u$. Similarly, we can obtain that $x^{\alpha_2}\partial_{y}F(u)=x^{\alpha_2}F'(u)\partial_{y}u$. Though simple verification, the result \eqref{202201081124} is obtained.
		\end{proof}
		
		\begin{lemma}\label{202301112303}
			Let $u^+=\max(u,0)$, $u^-=-\min(u,0)$. For any $u \in \mcH_0^1(\Omega)$, we have 
			\item (1)	
			\begin{equation}
				\partial_{x}u^+=
				\begin{cases}
					\partial_{x}u   & \  \mbox{a.e. on} \ \{u>0\},\\
					0,                     & \  \mbox{a.e. on}  \ \{u\le 0 \},
				\end{cases}
			\end{equation}
			
			\begin{equation}
				\partial_{x}u^-=
				\begin{cases}
					0,  & \  \mbox{a.e. on} \  \{u\ge 0 \},\\
					-\partial_{x}u ,                     & \  \mbox{a.e. on}  \ \{u< 0\},
				\end{cases}
			\end{equation}	
			
			\begin{equation}
			\partial_{y}u^+=
				\begin{cases}
				\partial_{y}u   & \  \mbox{a.e. on} \ \{u>0\},\\
					0,                     & \  \mbox{a.e. on}  \ \{u\le 0\},
				\end{cases}
			\end{equation}
			
			\begin{equation}
			\partial_{y}u^-=
				\begin{cases}
					0,  & \  \mbox{a.e. on} \  \{u\ge 0\},\\
					-\partial_{y}u ,     & \  \mbox{a.e. on}  \ \{u< 0\}.
				\end{cases}
			\end{equation}	
			
			\item (2) For any $u \in \mcH_0^1(\Omega)$, we have
			\begin{equation}
				Du=0  \ \mbox{ a.e. on the set} \ \{u=0\}.
			\end{equation}
			
			\item (3) If $u \in \mcH_0^1(\Omega)$, we have $|u| \in \mcH_0^1(\Omega)$.	  
		\end{lemma}
		
		\begin{proof}
			\item (1) For $r \in \mathbb{R}$ and $\epsilon >0$, let 
			\begin{equation}
				F_{\epsilon}(r)=
				\begin{cases}
					(r^2+\epsilon^2)^{\frac{1}{2}}-\epsilon,   & \ {r\ge 0},\\
					0,                     & \    {r< 0}.
				\end{cases}
			\end{equation}
			Then we find that $F_\epsilon \in C^1$ and $F_\epsilon'$ is bounded, and $u^+=\lim\limits_{\epsilon \to 0} F_{\epsilon}(u)$. By invoking the Lemma \ref{202301080919}, for any $\phi \in C_0^\infty(\Omega)$    we know 
			\begin{equation*}
				\iint_{\Omega} y^{\alpha_1} F_{\epsilon}(u) \partial_{x}\phi dxdy= -\iint_{\Omega} y^{\alpha_1} F_{\epsilon}'(u) \partial_{x}u  \phi dxdy,
			\end{equation*}
			letting $\epsilon \to 0$, utilizing the dominated convergence theorem,
			\begin{equation}\label{202301081232}
				\begin{split}	
					\iint_{\Omega} y^{\alpha_1} F_{\epsilon}(u) \partial_{x}\phi dxdy= \iint_{\Omega} y^{\alpha_1} u^+  \phi dxdy,\\
					-\iint_{\Omega} y^{\alpha_1} F_{\epsilon}'(u) \partial_{x}u  \phi dxdy= -\iint_{\Omega} y^{\alpha_1} \partial_{x}u \chi_{\{u>0\}} \phi dxdy,
				\end{split}
			\end{equation}
			so that 
			\begin{equation}
				y^{\alpha_1} \partial_{x} u^+ = y^{\alpha_1} \partial_{x}u \chi_{\{u>0\}} \ a.e..
			\end{equation}
			The rest of the cases are similar to the above process, we will not give a detailed description here.
			
			\item (2) From the result $(1)$, when $x \in  \{u\ge 0\}$, we have $Du^-=0$ a.e., when $x \in  \{u\le 0\}$, we have $Du^+=0$ a.e.. The set $\{u=0\}=\{u \ge 0\} \cap \{u \le 0\}$, and $Du=Du^+-Du^-=0$ on set $\{u=0\}$.
			
			\item (3) For any $u \in \mcH_0^1(\Omega)$, there is a sequence $(u_n)_{n \in \mathbb{N}} \subset C_0^\infty(\Omega)$ such that
			\begin{equation}\label{202302222249}
				\|u_n-u\|_{\mcH^1(\Omega)} \to 0, \ n \to \infty,
			\end{equation}
			then $u_n\in \mcH^1(\Omega)$, furthermore $u_n \in H_{\rm loc}^1(\Omega)$. Since $\supp u_n \subset \Omega$, we deduce that $u_n\in H_0^1(\Omega)$. Combining the results of $(1)$ and $(2)$, we know that $u_n^+, u_n^- \in H_0^1(\Omega)$. Thus there is a sequence  $\{v_n\}_{n \in \mathbb{N}} \subset C_0^\infty(\Omega)$ such that $\|v_n-u_n^+\|_{H^1(\Omega)} \to 0$. Given that 
			$u_n^+=\max\{u_n,0\}=\frac{1}{2}(u_n+|u_n|)$, $u^+=\max\{u,0\}=\frac{1}{2}(u+|u|)$, from \eqref{202302222249} we have 
	\begin{equation*}
		\begin{split}
			\|u_n^+-u^+\|_{\mcH^1(\Omega)} &\le \frac{1}{2} \|u_n-u\|_{\mcH^1(\Omega)} + \frac{1}{2} \Big{\|} \big| |u_n|-|u| \big| \Big{\|}_{\mcH^1(\Omega)}\\
			&\le \|u_n-u\|_{\mcH^1(\Omega)} \to 0, \ n\to \infty.
		\end{split}
	\end{equation*}	
	Hence,	
	\begin{equation*}
		\begin{split}
			\|v_n-u^+\|_{\mcH^1(\Omega)} &\le  \|v_n-u_n^+\|_{\mcH^1(\Omega)}+\|u_n^+-u^+\|_{\mcH^1(\Omega)}\\
			&\le \|v_n-u_n^+\|_{H^1(\Omega)}+\|u_n^+-u^+\|_{\mcH^1(\Omega)} \to 0, \ n\to \infty,
		\end{split}
	\end{equation*}
	we have $u^+\in \mcH_0^1(\Omega)$. Similarly, $u^-\in \mcH_0^1(\Omega)$. Consider that $|u|=u^+ + u^-$, then $|u|\in \mcH_0^1(\Omega)$.	
		\end{proof}

\begin{lemma}\label{202211101008}
	$(1)$ All the eigenvalues of $L$ is real and can be arranged in a monotone sequence on the basis of its (finite) multiplicity: 
	\begin{equation*}
		\begin{split}
\sigma(L)&=\{\lambda_k\}_{k=1}^{\infty},\quad 
\quad  0<\lambda_1 \le \lambda_2\le \lambda_3\le \cdots \le \lambda_k \cdots \rightarrow \infty, \quad k\rightarrow \infty.
		\end{split}		
	\end{equation*}
	
	$(2)$ There exists an  orthonormal basis $\{w_k\}_{k=1}^{\infty}\subset L^2(\Omega)$, where $w_k\in H_0(\Omega)$ is an eigenfunction with respect to $\lambda_k$, i.e., 
	\begin{equation*}
		\begin{cases}
			Lw_k=\lambda_kw_k, & x\in \Omega, \\
			u=0, & x\in\partial \Omega,
		\end{cases} 
	\end{equation*}
	for $k=1,2 \cdots. $
	
	$(3)$ We have
	\begin{equation}\label{202211051630}
		\lambda_k=\inf_{\substack{E \subset H_0(\Omega)\\  \dim(E)=k }} \sup_{\substack{u \in E \\ \|u\|_{L^2{(\Omega)}} = 1}} 	a(u,u).
	\end{equation}
	In particular, assuming that we have already computed $u_1,u_2,\cdots, u_{k-1}$ the $(k-1)\raisebox{0mm}{-}th$ first eigenfunctions, we also have:
	$\lambda_k=\inf\{a(u,u)\mid u\in H_0(\Omega), u\bot V_{k-1}, \|u\|_{L^2(\Omega)}=1\}$, where $V_{k-1}={\rm span}\{u_1,u_2\cdots,u_{k-1}\}$, the equality holds if and only if $u=w_k$.
	
	$(4)$ The eigenvalue $\lambda_1$ is simple and the first eigenfunction  $u_1$ is positive on $\Omega$.
\end{lemma}

\begin{proof}
	Employing standard functional analysis and compactness theory $(\mbox{remark} \ \ref{202211051625})$, the desired results $(1)$, $(2)$ and $(3)$ are simply acquired. From \eqref{202211051630} and Lemma \ref{202301112303}, $|u|$ is the eigenfunction for $\lambda_1$ if $u$ is. 
	Then  for any $\Omega' \subset \subset \Omega$,  we have $\sup\limits_{\Omega'}u \le C\inf\limits_{\Omega'}u$ for the equation $Lu-\lambda_1u=0$ by Theorem \ref{202210292207} . Since $|u|$ is non-negative in $\Omega$, we further obtain that $|u|$ is a positive (a.e.) eigenfunction. This indicates that the eigenfunctions of $\lambda_1$ are either positive or negative and thus it is impossible that two of them are orthogonal, hence $\lambda_1$ is simple.	
\end{proof}

From Lemma \ref{202211101008},  we obviously have $\lambda_1(V) \ge \lambda_1(0)$, so that $\inf\limits_{V \in S_p}\lambda_1(V)$ would be achieved by $V_0=0$. So, we are prefer to consider the problem $\sup\limits_{V \in S_p}\lambda_1(V)$, $1<p<\infty$.  Besides, 
\begin{equation}\label{202211101611}
	\begin{split}
\lambda_1(V)&=\inf_{\substack{u\in H_0(\Omega) \\  u \neq 0 }}	\frac{a_1(u,u)+\iint_\Omega Vu^2dxdy}{\iint_\Omega |u|^2dxdy}\\
&\le \  \frac{a_1(u,u)+\|V\|_{L^p(\Omega)} \|u\|_{L^{2q}(\Omega)}^2 }{\iint_\Omega |u|^2dxdy}\\
&\le  \  \frac{a_1(u,u)+ M \|u\|_{L^{2q}(\Omega)}^2 }{\iint_\Omega |u|^2dxdy},
	\end{split}
\end{equation}
where 
\begin{equation*}
	a_1(u,u)= \iint_\Omega|y^{\alpha_1}\pt_{x}u|^2+|x^{\alpha_2}\pt_{y}u|^2dxdy.
\end{equation*}
Setting 
\begin{equation*}
	J(u)=\frac{a_1(u,u)+ M \|u\|_{L^{2q}(\Omega)}^2 }{\iint_\Omega |u|^2dxdy}, \ u\neq 0, 
\end{equation*}
hence, if $u \in \mcH_0^1(\Omega) \cap  L^{2q}(\Omega) \subset H_0(\Omega)$, we immediately learn that, for all $V\in S_p$, $1<p<\infty$, 
\begin{equation}\label{202211122235}
\lambda_1(V) \le \inf_{\substack{u\in \mcH_0^1(\Omega) \cap  L^{2q}(\Omega)\\  u \neq 0 }}J(u).
\end{equation}

In the next work, for convenience we denote 
\begin{equation*}
	U=\mcH_0^1(\Omega) \cap  L^{2q}(\Omega),\quad L_1u=-\left(y^{2\alpha_1}\partial_{x}^2u+x^{2\alpha_2}\partial_{y}^2u\right), 
\end{equation*}
where $\frac{1}{p}+\frac{1}{q}=1$. 
 
\begin{lemma}\label{202211122210}
	The functional $J(u)$ attains its minimum in $\mcH_0^1(\Omega) \cap  L^{2q}(\Omega)$, further, the minimizers for $J(u)$ is non-negative.
\end{lemma}

\begin{proof}
	First of all, we notice that the functional $J(u)$ is not identically equal to $+\infty$. Let thus $\{u_n\}_{n \in N}$ be a minimizing sequence, i.e., 
	\begin{equation*}
		J(u_n) \downarrow \inf_{u\in U } J(u).
	\end{equation*}
So that we may assume the sequence is bounded in $U$, by Lemma \ref{Lem.20221006}, there is a subsequence $\{u_n\}_{n \in N}$, again denoted by itself, such that 
\begin{equation*}
	u_n \rightharpoonup z \  \mbox{in} \  U,
\end{equation*} 
 and by invoking Theorem \ref{202210121135}, we further  have 
  \begin{equation*}
  	u_n \to z \  \mbox{in} \   L^2(\Omega).
  \end{equation*}     
Utilizing the  lower semicontinuity
\begin{equation}
	J(z) \le \inf \limits_{u \in U} J(u).
\end{equation}
In summary, $z \in U$ and $J(z) \le \inf \limits_{u \in U} J(u)$, this proves that $J(z)=\inf \limits_{u \in U} J(u)$. In addition, by Lemma \ref{202301112303} we have $J(|u|)=J(u)$, we may assume that the minimizers are non-negative.
\end{proof}

\begin{lemma}\label{202211131121}
	Let $\tilde{u}$ ba a minimizer of $J(u)$, and assume that there exists  a function 	$\tilde{V} \in S_p$ with $1<p<\infty$ such that 
	\begin{equation}\label{202211102124}
	   L_1 \tilde{u}+ \tilde{V} \tilde{u} =\lambda \tilde{u}	
	\end{equation}
where $\lambda:= J(\tilde{u})$ is the minimum value of $J(u)$. Then 
\begin{equation*}
	\lambda_1(\tilde{V})=\lambda=J(\tilde{u}).
\end{equation*}	
\end{lemma}

\begin{proof}
	Without loss of generality, we may assume $\|\tilde{u}\|_{L^2(\Omega)}$=1. Assume that $v$ is the first eigenfunction with respect to $\tilde{V}$,  by Lemma \ref{202211101008} and 
	\begin{equation}\label{202211102125}
	L_1v+\tilde{V}v=\lambda_1 v.		
	\end{equation}
we have $v>0$ a.e. on $\Omega$ and $\lambda \ge \lambda_1$.

Suppose $\lambda_1 < \lambda$ by \eqref{202211122235}, considering \eqref{202211102124} and \eqref{202211102125}, 
\begin{equation*}
	a(\tilde{u},v) =\lambda   \iint_\Omega \tilde{u}v dxdy, \quad
	a(\tilde{u},v) =   \lambda_1 \iint_\Omega \tilde{u}v dxdy,
\end{equation*} 
thereby 
\begin{equation*}
	(\lambda_1-\lambda)\iint_\Omega \tilde{u}v dxdy=0.
\end{equation*}
This together with Lemma \ref{202211122210} tell us that $\tilde{u}v=0$ a.e., and given that $v>0 $ a.e.\! and $\tilde{u}$ is non-negative by Lemma \ref{202211122210} again, we obtain 
$\tilde{u}=0$ a.e., which is contradictory to the fact that $\|\tilde{u}\|_{L^2(\Omega)}=1$.	
\end{proof}

\begin{lemma}\label{202212261030}
 The function $J(\cdot)$ is Gateaux-differentiable, i.e., for any $\psi \in \mcH_0^1(\Omega) \cap  L^{2q}(\Omega)$ we have
\begin{equation*}
	J_{\psi}'(u)=\frac{2}{\|u\|_{L^2(\Om)}^2}\left(a_1(u,\psi)+M\left\|u^2\right\|_{L^q(\Omega)}^{1-q}\iint_{\Omega} |u|^{2q-2}u\psi dxdy-J(u)\iint_{\Omega} u\psi dxdy   \right).
\end{equation*}	
\end{lemma}

\begin{proof}
	Consider 
	\begin{equation}\label{202212261607}
\lim\limits_{t \to 0} \frac{J(u+t\psi)-J(u)}{t}=\lim\limits_{t \to 0}\frac{1}{t}\left(\frac{a_1(u+t\psi,u+t\psi)+M\|u+t\psi\|_{L^{2q}(\Omega)}^2}{\|u+t\psi\|_{L^2(\Omega)}^2}-\frac{a_1(u,u)+M\|u\|_{L^{2q}(\Omega)}^2}{\|u\|_{L^2(\Omega)}^2}\right),	
	\end{equation}
where 
\begin{equation}\label{202212261608}
\lim\limits_{t \to 0}\frac{1}{t}\left(\frac{a_1(u+t\psi,u+t\psi)}{\|u+t\psi\|_{L^2(\Omega)}^2}-\frac{a_1(u,u)}{\|u\|_{L^2(\Omega)}^2}\right)=\frac{2a_1(u,\psi)}{\|u\|_{L^2(\Omega)}^2}-\frac{2a_1(u,u)\iint_{\Omega}u\psi dxdy}{\|u\|_{L^2(\Omega)}^4}.
\end{equation}	
Moreover,
\begin{equation}\label{202212261609}
	\begin{split}	
&\lim\limits_{t \to 0}\frac{1}{t}\left(	\frac{M\|u+t\psi\|_{L^{2q}(\Omega)}^2}{\|u+t\psi\|_{L^2(\Omega)}^2}-\frac{M\|u\|_{L^{2q}(\Omega)}^2}{\|u\|_{L^2(\Omega)}^2}\right)\\
&=\lim\limits_{t \to 0} \frac{M\|u+t\psi\|_{L^{2q}(\Omega)}^2 \|u\|_{L^2(\Omega)}^2- M\|u\|_{L^{2q}(\Omega)}^2\|u\|_{L^2(\Omega)}^2}   {t \|u+t\psi\|_{L^2(\Omega)}^2  \|u\|_{L^2(\Omega)}^2}\\
&\quad -\lim\limits_{t \to 0}  \frac{2tM \|u\|_{L^{2q}(\Omega)}^2 \iint_{\Omega} u\psi dxdy+t^2M \|u\|_{L^{2q}(\Omega)}^2 \|\psi\|_{L^2(\Omega)}^2 }{t \|u+t\psi\|_{L^2(\Omega)}^2  \|u\|_{L^2(\Omega)}^2}\\
&=\lim\limits_{t \to 0} \frac{M\|u+t\psi\|_{L^{2q}(\Omega)}^2 - M\|u\|_{L^{2q}(\Omega)}^2} {t \|u+t\psi\|_{L^2(\Omega)}^2  }-\frac{2M \|u\|_{L^{2q}(\Omega)}^2 \iint_{\Omega} u\psi dxdy}{\|u\|_{L^2(\Omega)}^4}.
\end{split}
\end{equation}
It is not difficult to find that 
\begin{equation}\label{202212261610}
\lim\limits_{t \to 0} \frac{\|u+t\psi\|_{L^{2q}(\Omega)}^2- \|u\|_{L^{2q}(\Omega)}^2}{t}=2 \left( \iint_{\Omega}|u|^{2q}dxdy \right)^{\frac{1}{q}-1} \iint_\Omega |u|^{2q-2}u\psi dxdy.	
\end{equation}
Combining equations \eqref{202212261607} \eqref{202212261608} \eqref{202212261609} \eqref{202212261610}, we have
\begin{equation*}
	\begin{split}
		&\lim\limits_{t \to 0} \frac{J(u+t\psi)-J(u)}{t}\\
		&= \frac{2}{\|u\|_{L^2(\Omega)}^2} \left(a_1(u,\psi)-\iint_{\Omega}u\psi dxdy \frac{a_1(u,u)+M\|u\|_{L^{2q}(\Omega)}^2}{\|u\|_{L^2(\Omega)}^2}+ M\left\|u^2\right\|_{L^q(\Omega)}^{1-q}\iint_{\Omega} |u|^{2q-2}u\psi dxdy \right)\\
		&\triangleq (J'(u),\psi).		
	\end{split}
\end{equation*}
 Though simple calculation, we can verify that $(J'(u),\psi) \le C\|\psi\|_{\mcH(\Omega)}$ for any $\psi \in \mcH_0^1(\Omega) \cap  L^{2q}(\Omega)$, where $\|\cdot\|_{\mcH(\Omega)}=\|\cdot\|_{\mcH_0^1(\Omega)}+\|\cdot\|_{L^{2q}(\Omega)}$. Therefore, the desired result is proved.	
\end{proof}

\begin{theorem}\label{T4.5}
	Let $1<p< \infty$ and let us denote by $q$ its conjugate $(\frac{1}{p}+ \frac{1}{q}=1)$. Then the maximum of $\lambda_1(V)$ in the class $S_p$ is achieved by the function
	\begin{equation}\label{202212192128}
		V_0=M\left\|u_0^2\right\|_{L^q}^{1-q}|u_0|^{2(q-1)}
	\end{equation}
where $u_0$	is the minimizer in $\mcH_0^1(\Omega) \cap  L^{2q}(\Omega) \subset H_0(\Omega)$ of $J(u)$. The function $u_0$ can also be characterized as the first eigenfunction of the eigenvalue problem{\rm :}
\begin{equation}
	L_1 u+V_0u=\lambda_1(V_0)u,
\end{equation}
with $\lambda_1(V_0)=J(u_0)=\sup\limits_{V\in S_p} \lambda_1(V)$.	Furthermore, $V_0$ is the unique maximizer.
\end{theorem}

\begin{proof}
	From the Lemma \ref{202211122210}, $\inf\limits_{u \in U} J(u)$ is attained, recorded by $u_0$, and the minimizers is non-negative. Then taking the supremum in \eqref{202211122235} one obtains
	\begin{equation}\label{202211131122}
	\sup_{V\in S_p}	\lambda_1(V)\le J(u_0).
	\end{equation}

Consider $u_0 \in \mcH_0^1(\Omega) \cap  L^{2q}(\Omega)$, note that the function $J(\cdot)$ is Gateaux-differentiable by Lemma \ref{202212261030},  for any $\psi \in \mcH_0^1(\Omega) \cap  L^{2q}(\Omega)$ we have 
\begin{equation*}
	J_{\psi}'(u_0)=\frac{2}{\|u_0\|_{L^2(\Om)}^2}\left(a_1(u_0,\psi)+M\|u_0^2\|_{L^q(\Omega)}^{1-q}\iint_{\Omega} |u_0|^{2q-2}u_0\psi dxdy-J(u_0)\iint_{\Omega} u_0\psi dxdy   \right).
\end{equation*}
Hence a minimizer $u_0 \ge 0$ solves the equation $J_\psi'(u_0)=0$, i.e., 
\begin{equation*}
L_1 u_0+ V_0 u_0 =\lambda(V_0) u_0,\ u_0 \in \mcH_0^1(\Omega) \cap  L^{2q}(\Omega)
\end{equation*}
with $\lambda(V_0)=J(u_0)$ and $V_0=M\|u_0^2\|_{L^q}^{1-q} {|u_0|}^{2(q-1)}$, 
in other words,
\begin{equation}\label{202211132159}
a_1(u_0,\omega)+ \iint_{\Omega} V_0 u_0 \omega dxdy =\lambda(V_0) \iint_{\Omega} u_0 \omega dxdy \ \mbox{for any} \ \omega \in \mcH_0^1(\Omega) \cap  L^{2q}(\Omega).
\end{equation}

Through direct verification, $\|V_0\|_{L^p(\Omega)}=M$ can be obtained, so that $V_0 \in S_p$. Taking advantage of Lemma \ref{202211131121}  and the inequality \eqref{202211131122}, we readily find out that 
$\lambda_1(V_0)=J(u_0)=\sup\limits_{V\in S_p} \lambda_1(V)$.

Next we prove the uniqueness of the $V_0$, to show this, we require to first explain that $\lambda_1(V)$ is concave. Let 
\begin{equation*}
	R[u;V]=\frac{a_1(u,u)+\iint_\Omega Vu^2dxdy}{\iint_\Omega |u|^2dxdy},
\end{equation*}
recalling the definition of $\lambda_1$, for any $V_1, V_2 \in S_p$ and $0<t<1$,
\begin{equation*}
	\begin{split}
	\lambda_1(tV_1+(1-t)V_2)
	&= \inf_{\substack{u\in H_0(\Omega) \\  u \neq 0 }} \frac{a_1(u,u)+t\iint_\Omega V_1u^2dxdy+(1-t)\iint_\Omega V_2u^2dxdy}{\iint_\Omega |u|^2dxdy},\\
	&= \inf_{\substack{u\in H_0(\Omega) \\  u \neq 0 }} \left(tR[u;V_1]+(1-t)R[u;V_2]\right)
	\ge  t\lambda_1(V_1)+(1-t)\lambda_1(V_2).
	\end{split}
\end{equation*}
Suppose  that $V_1$ and $V_2$ are maximizing functions,  owing to the concavity we know that $V_3=\frac{1}{2}(V_1+V_2)$ is also a maximizing function. Let $u_1$, $u_2$ and $u_3$ denote their normalized first eigenfunctions, respectively. Clearly, unless $u_1=u_2=u_3$, since $\lambda_{1}$ is simple, we get  a contradiction in the following
\begin{equation*}
	\lambda_1^*=\lambda_1(V_3)=R[u_3;V_3]=\frac{1}{2}(R[u_3;V_1]+R[u_3;V_2])>\frac{1}{2}(\lambda_1(V_1)+\lambda_1(V_2))= \lambda_1^*. 
\end{equation*}
Hence, $u_1=u_2$. Consider that
\begin{equation*}
		L_1 u_1+V_1u_1=\lambda_1(V_1) u_1,\quad 
		L_1 u_2+V_2u_2=\lambda_1(V_2) u_2,
\end{equation*} 
we have $\iint_{\Omega} (V_1-V_2)u_1\varphi dxdy=0$ for all $\varphi\in C_0^\infty(\Omega)$, so that $V_1=V_2$ a.e. due to $u_1>0$. 
\end{proof}

\begin{theorem}
	Under the assumption of  Theorem \ref{T4.5}, the maximizing function $V_0$ and the corresponding first eigenfunction $u_0$ satisfy
	\begin{equation}
		\begin{split}
		\|u_0\|_{L^\infty(\Omega)}\le & \ \left(\frac{\lambda_1(V_0)}{M}   \right)^{\frac{p-1}{2}} \|u_0\|_{L^{2q}(\Omega)},	\\
		\|V_0\|_{L^\infty(\Omega)} \le & \ \lambda_1(V_0).		
		\end{split}
	\end{equation}
\end{theorem}        

\begin{proof}
	Let $c=\left(\frac{\lambda_1(V_0)}{M}   \right)^{\frac{p-1}{2}} \|u_0\|_{L^{2q}(\Omega)}$, and set $\zeta= u_0 -\min\{u_0,c\}$. Note that $\zeta\ge 0$ and $\zeta \in \mcH_0^1(\Omega) \cap  L^{2q}(\Omega)$,  and by \eqref{202211132159}
	\begin{equation}\label{202211132217}
		a_1(\zeta,\zeta)=a_1(u_0,\zeta)=\iint_{\Omega} \left(\lambda_1(V_0)-M\|u_0^2\|_{L^q}^{1-q}{|u_0|}^{2(q-1)} \right) u_0 \zeta dxdy.
	\end{equation}
If $\zeta>0$, that is, $u_0>c$, it is not hard to find that 
\begin{equation*}
	M\|u_0^2\|_{L^q}^{1-q}{|u_0|}^{2(q-1)} > M\|u_0^2\|_{L^q}^{1-q} c^{2(q-1)} =\lambda_1(V_0)
\end{equation*}
on the set $\{x\in\Omega\mid \zeta(x)>0\}$, this implies $a_1(\zeta,\zeta)<0$ from \eqref{202211132217}. However, $a_1(\zeta,\zeta)\ge 0$, which is a contradiction.  All these tell us that $\zeta =0$, furthermore, we have $u_0 \le c$. 

The estimate of $V_0$ \eqref{202212192128} can be obtained by applying the estimate of $u_0$. 	 
\end{proof}

\section{The fundamental gap}
The aim of this section is mainly to depict the optimal potentials over the class $S_p$, $p > 2$, some of theories are inspired by \cite{Ashbaugh1991ON,karaa1998extremal}. Before that, we first offer the existence for optimal function.

\begin{theorem}\label{T5.1}
The fundamental gap $\Gamma(V)$	attains its minimum in the classes of $S_p$ by $V^*$, $p >2$.
\end{theorem}

\begin{proof}
Firstly, we obtain that $\mcH_0^1(\Omega)$ is compactly embeded in $L^m(\Omega)$, $m=[1,4)$ from Theorem  \ref{202210121135}. Actually, at this situation,  we deduce that for all $u\in \mcH_0^1(\Omega)$ and $V \in L^p(\Omega)$, $p >2$, we have $\iint_{\Omega}Vu^2dxdy$ is bounded from $\|V\|_{L^p(\Omega)}\|u\|_{L^{2q}(\Omega)}^2$, where $2q=\frac{2p}{p-1}\in (2,4)$. Now utilizing the min-max formulae, it is easy to see that $\lambda_j(V)$ is uniformly bounded on $S_p$, $j=1,2$. Furthermore, we observe that
\begin{equation*}
	\|u\|_{\mcH_0^1(\Omega)} \le \|u\|_{H_0(\Omega)} \le (M+1)^{\frac{1}{2}}\|u\|_{\mcH_0^1(\Omega)}
\end{equation*}
by Theorem \ref{202210121135},
in other words, the norm $\mcH_0^1(\Omega)$ is equivalent to the norm $H_0(\Omega)$.
	 
Let $\{V^k\}_{k\in \mathbb{N}} \in S_p$ be the minimization sequence of $\Gamma(V)$, $p> 2$, i.e.,
\begin{equation*}
	\Gamma(V^k)\downarrow \inf\limits_{V\in S_p } \Gamma(V)= \Gamma^*. 
\end{equation*}
Let us denote by  $\lambda_j^k=\lambda_j(V^k)\ (j=1,2)$ and $u_j^k \ (j=1,2)$ the sequence of corresponding eigenvalues and eigenfunctions. As usual, $u_j^k$ is normalized by  $\|u_j^k\|_{L^2(\Omega)}= 1$. Up to a subsequence, for which we keep the index $k$, by the definition of  $S_p$ and \eqref{202211122235}, we can assume that
\begin{equation}\label{202212121516}
	V^k \rightharpoonup V^* \mbox{ in } L^p(\Omega), \quad  \lambda_j^k \to  \lambda_j^*.
\end{equation}		 
Now, since
\begin{equation*}
	\|u_j^k\|_{\mcH_0^1(\Omega)}^2 \le \iint_{\Omega} |y^{\alpha_1}{\partial_{x}u_j^k}|^2+|x^{\alpha_2}{\partial_{y}u_j^k}|^2+V^k|u_j^k|^2dxdy=\lambda_j^k \le C,
\end{equation*}
this implies that $u_j^k$ is also bounded in $\mcH_0^1(\Omega)$ and we can assume that
\begin{equation}\label{202212121517}
	u_j^k \rightharpoonup u^* \mbox{ in } \mcH_0^1(\Omega),
\end{equation}
and 
\begin{equation}\label{202212121518}
	u_j^k \to u^* \mbox{ in } L^{m}(\Omega), \ m \in [1,4).
\end{equation}

 For each $j$ we know for all  $v \in \mcH_0^1(\Omega)$                  
\begin{equation}\label{202212121152}
	{\iint_\Omega y^{2\alpha_1}\partial_{x}u_j^k\partial_{x}v +x^{2\alpha_2}\partial_{y}u_j^k\partial_{y}v+V^ku_j^k v dxdy} =\lambda_j^k \iint_\Omega u_j^k vdxdy, 
\end{equation}
choosing $k$ sufficiently large, according to \eqref{202212121516}, \eqref{202212121517},  and \eqref{202212121518}, hence
\begin{equation}\label{202212121527}
	{\iint_\Omega y^{2\alpha_1}\partial_{x}u_j^*\partial_{x}v +x^{2\alpha_2}\partial_{y}u_j^*\partial_{y}v+V^*u_j^* vdxdy} =\lambda_j^* \iint_\Omega u_j^* vdxdy. 
\end{equation}
This shows that $\lambda_j^k$ converges to an element of the spectrum of the problem \eqref{202301172048} given by $V^*$. Especially, we may extract a subsequence of  $\{u_1^k\}_{k\in \mathbb{N}}\subset L^2(\Omega)$ such that $u_1^k$ is converge to $u_1^*$ a.e. \!by \eqref{202212121518}, owing to the non-negativity of $u_1^k$ (see Lemma \ref{202211101008} (4)), then $u_1^*$ must be the first eigenfunction and $\lambda_1^*=\lambda_1(V^*)$. Note that $u_2^*$ and $u_1^*$ are orthogonal on $\Omega$, we know that $u_2^*$ must change the sign on $\Omega$, which indicates that  $u_2^*$ not to be the first eigenfunction of $\lambda_1(V^*)$ by Lemma \ref{202211101008}. Hence, we know $\lambda_2^*\ge \lambda_{2}(V^*)$, furthermore,
\begin{equation}
	\Gamma(V^k) \rightarrow \lambda_2^*-\lambda_1^*\ge \Gamma(V^*), k \rightarrow \infty,
\end{equation}
that is, $\Gamma(V^*) \le \inf\limits_{V\in S_p } \Gamma(V)$. Meanwhile,  $\Gamma(V^*) \ge \inf\limits_{V\in S_p } \Gamma(V)$, then  we see that  the minimum value  of $\Gamma(V)$ can be attained on $S_p$, $p >2$.
\end{proof}

\begin{definition}[\cite{Ashbaugh1991ON}]\label{202301161609}
	A real-valued, measurable and bounded  function $P(x)$ on $\Omega$ is called an {\it admissible perturbation} of $V(x)$ if and only if ${\rm dist}(V+tP,S_p)=o(t)$, i.e.,
	\begin{equation*}
		\inf_{\hat V\in S_p}\|V+tP-\hat V\|=o(t). 
	\end{equation*} 
It is {\it strongly admissible}  if and only if $V(x)+tP(x) \in S_p$ for any sufficiently small $t$, where $t$ can only be non-negative or non-positive, or it can be any sign. An admissible perturbation is thus either strongly admissible or tangential to $\partial{S_p}$.
A perturbations admissible for both positive and negative small $t$ are {\it tangential} to $\partial{S_p}$ in $L^p$ $(p>2)$ in the sense that 
\begin{equation}\label{202212132213}
	\iint_{\Omega}V^{p-1}P(x)dxdy=0.
\end{equation}
\end{definition}

\begin{remark}
Define the functional $\rho:L^p(\Omega) \to \mathbb{R}^+$ by $\rho(V)=\|V\|_{L^p(\Omega)}^p$, $p>2$, the set  $\partial{S_p}$ is a level surface of $\rho$. The functional $\rho$ is Fr\'echet differentiable \cite[Proposition 1.12]{Minimax} and 
\begin{equation}\label{202301121057}
d_V\rho(P):=(\rho'_V,P)=p\iint_\Omega |V|^{p-2}VPdxdy
\end{equation}
for any $P \in L^p(\Omega)$.
Referring  to \cite[(p153-154)]{poschel1987inverse}, we know the surface  $\partial{S_p}$ is real analytic submanifold with tangent space given by 
 	\begin{equation*}
 		T_V\partial{S_p}=Ker d_V \rho,
 	\end{equation*}
together with \eqref{202301121057}, we have 
 \begin{equation*}
	T_V\partial{S_p}=\big\{P \in L^p(\Omega): \iint_\Omega V^{p-1}Pdxdy=0\big\}.
\end{equation*}
\end{remark}

\begin{lemma}[\cite{hislop2012introduction,bookPerturbation,pankov2006introduction,RichardCourant}]\label{202304021038}
If the $\lambda_2(V)$ is simple, for any admissible perturbation $P(x)$, we have    
	\begin{equation}\label{202212121849}
		\left.\frac{d\Gamma(V+tP(x))}{dt}\right|_{t=0}=\iint_{\Omega} P(x)\left(|u_2|^2-|u_1|^2 \right) dxdy,
	\end{equation}
where  $u_i$ is a normalized eigenfunction associated to $\lambda_{i}(V)$, $i=1,2$.
Suppose $\lambda_2(V)$ were $r$ fold degenerate, then for any admissible perturbation $P(x)$, $\lambda_2$ can split into a cluster of eigenvalues $\lambda_{2,m}$, $m=1,2, \cdots, r$, which can be considered as a set of differentiable functions near $t=0$, but those functions do not ordinary correspond to the ordering of eigenvalues given by the min-max principle. For example, the lowest one for $t<0$ will typically be the highest for $t>0$.  And 
\begin{equation}\label{202212141557}
	\left.\frac{d\Gamma_m(V+tP(x))}{dt}\right|_{t=0}=\iint_{\Omega} P(x)\left(|u_{2,j}|^2-|u_1|^2 \right) dxdy,
\end{equation}
where $\Gamma_m=\lambda_{2,m}-\lambda_1$, and the orthonormal eigenfunctions $u_{k,j}$ are chosen so that
\begin{equation*}
\iint_{\Omega} u_{2,j}P u_{2,m}dxdy=0, \ j \neq m. 
\end{equation*}
\end{lemma}

\begin{proof}
 We will follow the idea of \cite[p343-348]{RichardCourant} to prove the first order perturbation of the eigenvalue. Consider the operator $L=-\left(y^{2\alpha_1}\partial_{x}^2u+x^{2\alpha_2}\partial_{y}^2u\right)+V$, $V \in S_p$, $p>2$, note that $Dom(L)=\mcH_0^1(\Omega)$ is dense in $L^2(\Omega)$, and for any $u, v \in \mcH_0^1(\Omega)$,
\begin{equation*}
	(Lu,v)=(v,Lu),
\end{equation*}
it implies that $L$ is a Hermitian operator. Similarly, for $L(t)=L+tP$, where $P$ is an admissible perturbation as defined in Definition \ref{202301161609}, we observe that $L(t)$ is a Hermitian operator for real small $|t|<\epsilon$. Besides, for  $u \in Dom(L^*)$, then there exists $u^* \in L^2(\Omega)$ such that for any  $v \in \mcH_0^1(\Omega)$, $(u^*,v)=(u,Lv)$, where $L^*$ is the conjugate operator of $L$. Since the range of operator $L$ is $L^2(\Omega)$ (Theorem \ref{202301161746}), there exists a $w \in \mcH_0^1(\Omega)$ such that $u^*=Lw$. It is not hard to see that  $(w,Lv)=(Lw,v)=(u^*,v)=(u,Lv)$, then $u=w \in Dom(L)$. Furthermore, $Dom(L) \subset Dom(L^*)$, we obtain $Dom(L)=Dom(L^*)$. All these reveal that $L$ is self-adjoint. In addition, $L$ is also a closed operator. For $u_n \in \mcH_0^1(\Omega)$, $u_n \to u$ and $Lu_n \to g$ in $L^2(\Omega)$, we will show that $u \in \mcH_0^1(\Omega)$ and $g=Lu$. Indeed, for any $\phi \in C_0^\infty(\Omega)$,
 \begin{equation*}
 	(Lu_n,\phi)=(u_n,L\phi)\to (u,L\phi)=(Lu,\phi),
 \end{equation*} 
 given that $Lu_n \to g$,  $(Lu_n,\phi) \to (g,\phi)$, we obtain $Lu=g$ in $L^2(\Omega)$ and $u \in\mcH_0^1(\Omega)$. Thus the operator $L(0)=L$ on $\mcH_0^1(\Omega)$ is hypermaximal. Consider $P \in L^\infty(\Omega)$, then the operator $L(t)$ on $\mcH_0^1(\Omega)$ is regular in a real neighborhood of $\epsilon=0$ by Criterion 3 \cite[p78]{rellich1969perturbation}. 
 
 From Lemma \ref{202211101008} and the Theorem 3 \cite[p74]{rellich1969perturbation}, we know if $\lambda$ is an eigenvalue of finite multiplicity $r$ of the operator $L$, then there exists power series \cite[p54]{rellich1969perturbation}
\begin{equation*}
	\lambda^{1}(t),\cdots,\lambda^{r}(t)
\end{equation*}
and power series in Hilbert space
\begin{equation*}
u^{1}(t),\cdots,u^{r}(t)
\end{equation*}
all convergent in a neighborhood of $t=0$, which satisfy
\begin{equation}\label{202301171708}
	L(t)u^{i}(t)=\lambda^{i}(t)u^{i}(t), \ i=1,\cdots,r,
\end{equation}
 $\lambda^{i}(0)=\lambda$, $ i=1,\cdots,r$, and $(u^{i}(t),u^{j}(t))=\delta_{i,j}$, $i,j=1,\cdots,r$.  Then we know that if $\lambda_k$ is simple, $k=1,2$,  we may assume
 \begin{equation*}
 	\begin{split}
 \lambda_{k}(t)&=\lambda_{k}+t\mu_{k,1}+t^2\mu_{k,2}+\cdots,\\
  u_{k}(t)&=u_{k}+tv_{k,1}+t^2v_{k,2}+\cdots.		
 	\end{split} 	
 \end{equation*}
Substitute them into \eqref{202301171708}, and given that $Lu_k=\lambda_{k}u_k$, we obtain that
\begin{equation}
	Lv_{k,1}+Pu_k=\mu_{k,1}u_k+\lambda_{k}v_{k,1},
\end{equation}
we multiply above equation by $u_l$ and integrate on $\Omega$, hence
\begin{equation}
\lambda_{l}	\iint_{\Omega} u_l v_{k,1}dxdy -\lambda_{k}	\iint_{\Omega} u_l v_{k,1}dxdy=\mu_{k,1} \iint_{\Omega} u_l u_{k}dxdy -\iint_{\Omega} P u_l u_k dxdy.
\end{equation}
Letting $k=l$, we obtain the perturbations of first order $\mu_{k,1}=\iint_\Omega Pu_k^2dxdy$.

If the eigenvalue $\lambda_2(V)=\lambda$ were $r$ fold degenerate, i.e.,  $\lambda_{1}<\lambda_{2}=\lambda_{3}=\cdots=\lambda_{r+1}<\lambda_{r+2}$, we may assume the $r$ eigenfunctions associated with the eigenvalue $\lambda$ transformed into another system of such eigenfunctions by means of an orthogonal transformation 
\begin{equation}
	u_n^*=\sum_{j=2}^{r+1}\gamma_{n,j}u_j,  \  \ \ (n=2,\cdots,r+1)\\
\end{equation}
which will be determined later. We now assume for $n=2,,\cdots,r+1$, 
\begin{equation*}
	\begin{split}
		\lambda_{n}(t)&=\lambda_{k}+t\mu_{n,1}+t^2\mu_{n,2}+\cdots,\\
			u_n(t)&=u_n^*+tv_{n,1}+t^2v_{n,2}+\cdots,		
	\end{split}
\end{equation*}
Similarly, by $Lu_n^*=\lambda_{n}u_n^*$ and \eqref{202301171708}, 
\begin{equation*}
	Lv_{n,1}+Pu_n^*=\mu_{n,1}u_n^*+\lambda_{n}v_{n,1},
\end{equation*}
multiply above equation by $u_l$ and integrate on $\Omega$, hence
\begin{equation*}
(\lambda_{l}-\lambda_{n})\iint_{\Omega}	u_l v_{n,1}dxdy =\sum_{j=2}^{r+1}\gamma_{n,j} \left(\iint_\Omega  
\mu_{n,1}u_ju_l-Pu_ju_ldxdy\right),
\end{equation*}
and hence in particular for $l=2,\cdots, r+1$ we know
\begin{equation*}
	0=\sum_{j=2}^{r+1}(d_{j,l}-\mu_{n,1}\delta_{j,l})\gamma_{n,j}, \ \ \ (l,n=2,\cdots,r+1),
\end{equation*}	
where $d_{j,l}=\iint_\Omega Pu_ju_ldxdy$. From these $r^2$ equations, $\mu_{n,1}$, $n=2,\cdots,r+1$, may be uniquely determined as roots of the characteristic equation $|d_{j,l}-\mu_{n,1}\delta_{j,l}|=0$ \cite[Chapter I]{RichardCourant}. For simplicity, select a system of orthonormal eigengunctions $u_2,\cdots,u_{r+1}$ for which $d_{j,l}=0$, $j \neq l$. Then the orthogonal matrix $(\gamma_{n,j})$ is also determined uniquely, and the $u_n^*$ is now fixed. Let us designate these  $u_n^*$ by $u_n$. The matrix $(d_{j,l})$ in the new notation is a diagonal matrix with the elements 
$d_{n,n}=\mu_{n,1}$, the remaining elements are zero. Then the first order perturbation of the eigenvalue is given by $\mu_{n,1}=d_{n,n}$.
\end{proof}

\begin{theorem}
	Let  $V^*\in S_p\ (p> 2)$ be a 	minimizer of $\Gamma(V)$ for the eigenvalue problem \eqref{202301172048}. Then
	
	\item(1) $\lambda_2(V^*)$ is non-degenerate.
	
	\item(2) $\supp V^* \subsetneq \Omega$ and $\|V^*\|_{L^p(\Omega)}=M$, where $\supp V^*\subsetneq \Omega$ means $|\Omega-\supp V^*|>0$.   
	
	\item(3) Moreover, 
	\begin{alignat*}{2}
			|u_2^*|^2-|u_1^*|^2 
			&\ge 0 \   && \mbox{\rm on} \  \Omega \backslash \supp V^*,\\
			|u_2^*|^2-|u_1^*|^2
			&=c(V^*)^p\  \  &&  \mbox{\rm a.e.\! on} \ \supp V^* ,				
		\end{alignat*}
for some constant $c<0$, where  $u_1^*$ and $u_2^*$ are the first and second normalized eigenfunctions with respect to $V^*$, respectively.  
\end{theorem}
	 
\begin{proof}
{\it Step 1}. Suppose $\lambda_2(V^*)$ is simple. We claim that the set $Q=\{x \in \Omega:|u_2^*|^2=|u_1^*|^2\}$ can only happen on a set of zero. 

By contradictory, without losing generality, we suppose the set $Q^+=\{x \in Q: u_2^*>0\}$ is of positive measure,  we know $Q=Q^+ \cup Q^-$ in view of that $u_1^*>0$ on $\Omega$, where $Q^-=\{x \in Q: u_2^*<0\}$.  

Clearly, $u_1^*-u_2^*=0$ a.e.\! on $Q^+$ and $u_1^*-u_2^* \in H_{\rm loc}^1(\Omega)$, denote $u_e=u_1^*-u_2^*$, we have $Du_e=0$ a.e. on $Q^+$ by \cite{gilbarg2015elliptic}. Since $u_e \in H_{\rm loc}^2(\Omega)$ by Theorem \ref{201301111756}, we obtain that $\Delta u_e =0$ a.e. on $Q^+$ by Lemma \ref{202301112303}.  Therefore $y^{2\alpha_1}\partial_{x}^2 u_e+ x^{2\alpha_2}\partial_{y}^2 u_e =0$ a.e. on $Q^+$,  substituting it into the eigenvalue problem \eqref{202301172048}, we see that $(\lambda_2-\lambda_1)u_1^*=0$ a.e. on $Q^+$, which is impossible since $u_1>0$ on $\Omega$.

{\it Step 2}. We exclude the situation of $\|V^*\|_{L^p(\Omega)}=0$, i.e., $V=0$ a.e.\! on $\Omega$. 

By contradictory,  we assume $V^*=0$ a.e.\! on $\Omega$. Let $x_0$  be any point in $\Omega$ and $G_j \subset \Omega$ be any measurable sequence of subsets containing $x_0$, then, perturbations of the form $P(x)=\chi_{G_j}$ are  admissible for small $t\ge 0$. Then we see that
 \begin{equation}\label{202212121850}
	\left.\frac{d\Gamma(V^*+tP(x))}{dt}\right|_{t=0}=\iint_{G_j} |u_2^*|^2-|u_1^*|^2 dxdy \ge 0.
\end{equation}
Dividing the $|G_j|$ and letting $G_j$ shrink nice to $x_0$ as $j \to \infty$,  from the Lebesgue Differential Theorem we have $|u_2^*|^2-|u_1^*|^2 \ge 0$  on $\Omega$, it follows from the Step 1 that $\iint_{\Omega}|u_2^*|^2-|u_1^*|^2 dxdy > 0$. However, this is contrary to the conditions of normalization $\iint_\Omega |u_2^*|^2-|u_1^*|^2dxdy=0$.

{\it Step 3}. Now, we will rule out the possibility that  $0<\|V^*\|_{L^p}< M$. 
Suppose $0<\|V^*\|_{L^p}< M$, consider the set $R=\supp V^*$ (in the sense of distribution), then for any $x_0 \in R$, and $R_j \subset R$ be any measurable sequence of subsets containing $x_0$, in fact, $P(x)=\chi_{R_j}$ are  admissible for small $t\in (-\epsilon,\epsilon)$, it is not hard to discover that
\begin{equation}\label{202212131008}
	\left.\frac{d\Gamma(V^*+tP(x))}{dt}\right|_{t=0}=\iint_{G_j} |u_2^*|^2-|u_1^*|^2dxdy = 0.
\end{equation} 
Once again dividing by $|R_j|$ and letting $R_j$ shrink nice to $x_0$ as $j \to \infty$, we find by the Lebesgue Differential Theorem that $|u_2^*|^2-|u_1^*|^2=0$ on $R$. This reveals that $R$ is a zero measure subset by the conclusion of Step 1. One can  immediately receive that $V^*=0$ a.e., this is inconsistent with the fact $\|V^*\|_{L^p}> 0$.

{\it Step 4}. Based on the above results, we conclude that $\|V^*\|_{L^p(\Omega)}=M$. Now, we prove that the set $\Omega \backslash \supp V^*$ must be a positive measure set. 

Otherwise, we assume $V>0$ a.e.\! on $\Omega$. For $x_0\in \Omega$, and $H_j \subset \Omega$ be any measurable sequence of subsets containing $x_0$, $P(x)=\chi_{H_j}$ are  admissible for small $t \le 0$, 
\begin{equation}\label{202212131054}
	\left.\frac{d\Gamma(V^*+tP(x))}{dt}\right|_{t=0}=\iint_{H_j} |u_2^*|^2-|u_1^*|^2 dxdy \le  0,
\end{equation} 
so that $|u_2^*|^2-|u_1^*|^2 \le 0$ on $\Omega$. Combined with the conclusion of Step 1 again, we have  $\iint_{\Omega}|u_2^*|^2-|u_1^*|^2 dxdy < 0$, this is contrary to the normalization condition. This implies that $\supp V^* \subsetneq \Omega$. 

{\it Step 5}.  We show that 
\begin{equation}\label{202212132244}
	|u_2^*|^2-|u_1^*|^2=c(V^*)^{p-1} \ \mbox{a.e. on }  \supp V^*,	
\end{equation} 
here $c<0$ is a constant and $V^*$ is continuous a.e..

On one hand, let us consider perturbations (tangential to $\partial{S_p}$) of the form 
\begin{equation}\label{202212132214}
	P(x)=\frac{\chi_{T_1}}{\iint_{T_1} (V^*)^{p-1}dxdy}- \frac{\chi_{T_2}}{\iint_{T_2} (V^*)^{p-1}dxdy},
\end{equation}
where $T_1$ and $T_2$ are disjoints subsets lie in $\supp V^*$, we observe that $P(x)$ as in \eqref{202212132214} is admissible by \eqref{202212132213}  for both positive and negative small $t$. Hence,
\begin{equation}\label{202212132217}
\left.\frac{d\Gamma(V^*+tP(x))}{dt}\right|_{t=0}= \frac{ \iint_{T_1}|u_2^*|^2-|u_1^*|^2dxdy} {\iint_{T_1} (V^*)^{p-1} dxdy}- \frac{ \iint_{T_2}|u_2^*|^2-|u_1^*|^2dxdy} {\iint_{T_2} (V^*)^{p-1} dxdy} = 0 
\end{equation} 
for all admissible sets $T_1,T_2$, from which we have \eqref{202212132244} is established.

On the other hand,  we can now prove that 
\begin{equation}\label{202302211744} 
	|u_2^*|^2-|u_1^*|^2 \le 0 \ \mbox{a.e. on }  \supp V^*,	
\end{equation}
utilizing the same strong perturbation argument as \eqref{202212131054}.

In conjunction with the result \eqref{202212132244},\eqref{202302211744}, we have  $c<0$, and $V^*$ is continuous a.e. since $u_1^*$ and $u_2^*$ are locally continuous on $\Omega$ by Theorem \ref{201301111756}. 

{\it Step 6}. We will illustrate that $|u_2^*|^2-|u_1^*|^2 \ge 0$ on $\Omega \backslash \supp V^*$.  

Indeed, let $x$  be any point in $\Omega \backslash \supp V^*$ and $F_j \subset \Omega \backslash \supp V^*$ be any measurable sequence of subsets containing $x$, then $P(x)=\chi_{F_j}$ are  admissible for small $t\ge 0$ based on \eqref{202212132213}, according to the same theory as before, we see that $|u_2^*|^2-|u_1^*|^2 \ge 0$ on $\Omega \backslash \supp V^*$.

{\it Step 7}.  We prove  $\lambda_2(V^*)$ can not be degenerate. 
Suppose $\lambda_2(V^*)$ are $r$ fold degenerate, by Lemma \ref{202304021038}, then for any admissible pertubation $P(x)$, the cluster of eigenvalues {$\lambda_{2,m}(t)$} into which $\lambda_{2}(V^*)$ would split could be arranged to be analytic in $t$
at $t=0$, and likewise for the associated orthonormalized eigenfunctions $\{u_{2,m}\}$ (depending on $P$).

Suppose now $u^*$ is any normalized vector in the eigenspace for $\lambda_{2}$, then
\begin{equation*}
	u^*=\sum_{m=1}^{r}c_mu_{2,m}^*, \quad  \sum_{m=1}^{r}|c_m|^2=1,
\end{equation*}
where $c_m \in \mathbb{R}$, and 
\begin{equation}\label{202212141827}
	\begin{split}
\iint_{\Omega} P(x)\left(|u^*|^2-|u_1^*|^2\right)dxdy &=	\iint_{\Omega} P(x) \left(\sum_{m=1}^{r}  |c_m|^2|u_{2,m}^*|^2-|u_1^*|^2 \right) dxdy\\
&=\sum_{m=1}^{r}  |c_m|^2 \iint_{\Omega} P(x) (|u_{2,m}^*|^2-|u_1^*|^2)dxdy.		
	\end{split}
\end{equation}

As we did in the Step 2 and Step 3, we can remove  the possibility that  $0 \le \|V^*\|_{L^p} <M$. More precisely, 

(a) Suppose $\|V^*\|_{L^p}=0$, if $P(x)$ is a positive perturbation for small $t \ge 0$, then we must have 
\begin{equation*}
\left.\frac{d\Gamma_m}{dt}\right|_{t=0}\ge 0	
\end{equation*}
for any $m$. Otherwise,  suppose $\frac{d\Gamma_m}{dt}\Big|_{t=0}<0$, we would have 
\begin{equation}\label{202212151856}
\Gamma(t_0) \le \Gamma_m(t_0) <\Gamma_m(0)= \Gamma(0)	
\end{equation}
 for some  $t_0 >0$, this contradicts the fact that $V^*$ is a minimum. Coupling with \eqref{202212141827}, we find 
\begin{equation*}
\iint_{\Omega} P(x)\left(|u^*|^2-|u_1^*|^2\right)dxdy \ge 0,
\end{equation*}
as discussed as before, we have $|u^*|^2-|u_1^*|^2 \ge 0$ on $\Omega$.  This will cause the same contradiction in Step 2. 

(b) Suppose $0<\|V^*\|_{L^p}<M$, consider the set $R=  \supp V^* $, we have 
\begin{equation}\label{202212142029}
	\left.\frac{d\Gamma_m}{dt}\right|_{t=0}= 0	
\end{equation}
for all admissible perturbations $P(x)$ as $t \in (-\epsilon,\epsilon)$ and all $m$, where $\supp P(x) \subset R$. Otherwise, if  
$\frac{d\Gamma_m}{dt}\Big|_{t=0}> 0$, then would lead to the same contradiction for some $t_0<0$ as \eqref{202212151856}. If $\frac{d\Gamma_m}{dt}\Big|_{t=0}<0$, we would find the result \eqref{202212151856} is established for some $t_0>0$, which is  impossible. Based the above information, we have 
\begin{equation*}
	\iint_{R} P(x)(|u^*|^2-|u_1^*|^2)dxdy=0.
\end{equation*}
Repeating the Step 3, we know that this case is excluded.

{\it Step 8}. We argue as Step $4$-Step $6$, we conclude that 
\begin{equation}\label{202212151109}
	\begin{split}
		|u^*|^2-|u_1^*|^2 \le 0 \ &\mbox{on}  \supp V^*,\\
		 |u^*|^2-|u_1^*|^2 \ge 0 \  &\mbox{on} \  \Omega \backslash \supp V^*.
	\end{split}
 \end{equation}
  And suppose there are two orthonormal vectors, $u_{2,a}$ and $u_{2,b}$ in the second eigenspace, and that $x_0$ is a point on $\partial{B} \cap \Omega$, $B=\supp V^*$, so that we may take 
$u_{2,a}^*(x_0)=u_{2,b}^*(x_0)=u_{1}(x_0)^* \neq 0$, then the normalized  eigenfunction 
\begin{equation*}
	u^*(x)=\frac{1}{\sqrt{2}}(u_{2,a}^*(x)-u_{2,b}^*(x))
\end{equation*}
is zero at $x=x_0$, then $u^*(x_0)<u_1^*(x_0)$. Because of the continuity for eigenfunctions, it would follow that $u_1^*(x)>u^*(x)$ on part of $\Omega\backslash B$,  which is contrary to \eqref{202212151109}.
\end{proof}

\noindent{\bf Acknowledgement}

\vspace{2mm}

This work is supported by the National Natural Science Foundation of China, the Science-Technology Foundation of Hunan Province.

\normalem

\bibliographystyle{plain}
\bibliography{ref.bib}

\end{document}